\documentclass[reqno]{amsart}
\usepackage[foot]{amsaddr}
\usepackage{amssymb,amsmath,amsfonts,mathptmx,cite,amsthm,mathrsfs,url}
\usepackage{tikz-cd}

\usepackage{rotating}

\theoremstyle{plain}
\newtheorem{theorem}{Theorem}[section]
\newtheorem{proposition}[theorem]{Proposition}
\newtheorem{lemma}[theorem]{Lemma}
\newtheorem{corollary}[theorem]{Corollary}

\theoremstyle{remark}
\newtheorem*{remark}{Remark}
\newtheorem{example}[theorem]{Example}

\DeclareSymbolFont{rsfscript}{OMS}{rsfs}{m}{n}
\DeclareSymbolFontAlphabet{\mathrsfs}{rsfscript}

\newcommand{\mA}{\mathcal{A}}

\DeclareMathOperator{\alf}{alph}
\DeclareMathOperator{\var}{var}

\numberwithin{equation}{section}

\newcommand{\ais}{ai-semi\-ring}

\def\fb{finitely based}
\def\nfb{non\-finitely based}

\begin{document}
\title[Identities of triangular Boolean matrices]{Identities of triangular Boolean matrices}

\author[M. V. Volkov]{Mikhail V. Volkov}
\address{{\normalfont 620075 Ekaterinburg, Russia}}
\email{m.v.volkov@urfu.ru}

\begin{abstract}
We give a combinatorial characterization of the identities holding in the semiring of all upper triangular Boolean $n\times n$-matrices and apply the characterization to computational complexity of identity checking, finite axiomatizability of equational theories, and algebraic descriptions of certain classes of recognizable languages.
\end{abstract}

\maketitle

\section*{Introduction}
A \emph{Boolean} matrix is a matrix with entries $0$ and $1$ only. The addition and multiplication of such matrices are as usual, except that addition and multiplication of the entries are defined as $x+y:=\max\{x,y\}$ and $x\cdot y:=\min\{x,y\}$. (Here and below the sign $:=$ stands for equality by definition; thus, $A:=B$ means that $A$ is defined as $B$.) For each $n$, the set of all Boolean $n\times n$-matrices forms an additively idempotent semiring under matrix addition and multiplication.  Here an \emph{additively idempotent semiring} (\ais, for short) is an algebra $(S,+,\cdot)$ with binary addition $+$ and multiplication $\cdot$ such that $(S,+)$ is a commutative and idempotent semigroup, $(S,\cdot)$ is a semigroup, and multiplication distributes over addition on the left and right.

A Boolean matrix $\bigl(\alpha_{ij}\bigr)_{n\times n}$ is called \emph{upper triangular} if $\alpha_{ij}=0$ for all $1\le j<i\le n$. The set $T_n$ of all upper triangular Boolean $n\times n$-matrices is closed under matrix addition and multiplication so it forms a subsemiring of the \ais{} of all  Boolean $n\times n$-matrices. Our aim is to characterize the semiring identities of the \ais\ $(T_n,+,\cdot)$. As a consequence, we get a characterization of the semigroup identities of the semigroup $(T_n,\cdot)$.

Our characterizations of identities of $(T_n,+,\cdot)$ and $(T_n,\cdot)$ are presented in Sect.~\ref{sec:main}; their proofs only use notions detailed in Sect.~\ref{sec:prelim} and the addition and multiplication rules for $n\times n$-matrices:
\[
\bigl(\alpha_{ij}\bigr)_{n\times n}+\bigl(\beta_{ij}\bigr)_{n\times n}:=\bigl(\alpha_{ij}+\beta_{ij}\bigr)_{n\times n}, \qquad \bigl(\alpha_{ij}\bigr)_{n\times n}\cdot\bigl(\beta_{ij}\bigr)_{n\times n}:=\left(\sum_{k=1}^n\alpha_{ik}\beta_{kj}\right)_{n\times n}.
\]
Section~\ref{sec:applications} contains applications of our main results. The applications deal with the computational complexity of identity checking, finite axiomatizability of equational theories, and algebraic descriptions of certain classes of recognizable languages. While Sects.~\ref{sec:prelim} and~\ref{sec:main} are truly self-contained, understanding material of Sect.~\ref{sec:applications} requires some background specified at the beginning of subsections.

\section{Preliminaries}
\label{sec:prelim}
To state our results, we need several combinatorial and algebraic concepts. Most are standard, but some are less known or even new. We provide a self-contained introduction for the reader's convenience. When doing so, we also set up our notation.

\subsection{Words and subwords}
A (\emph{semigroup}) \emph{word} is a finite sequence of symbols, called \emph{variables}. Sometimes we employ the \emph{empty word}, that is, the empty sequence. In particular, we adopt the convention that for a variable $x$
and a nonnegative integer $m$, the expression $x^m$ denotes $\underbrace{x\cdots x}_{\text{$m$ times}}$ if $m>0$ and the empty word if $m=0$.

We denote words by lowercase boldface letters. If $\mathbf{w}=w_1\cdots w_k$, where $w_1,\dots,w_k$ are variables, possibly with repeats, then the set $\{w_1,\dots,w_k\}$ is denoted by $\alf(\mathbf{w})$ and called the \emph{alphabet} of $\mathbf{w}$ and the number $k$ is denoted by $|\mathbf{w}|$ and called the \emph{length} of $\mathbf{w}$. If $\mathbf{w}$ is empty, then we define $\alf(\mathbf{w}):=\varnothing$ and $|\mathbf{w}|:=0$.

Words are multiplied by concatenation: the product of two words $\mathbf{w}=w_1\cdots w_k$ and $\mathbf{w}'=w'_1\cdots w'_\ell$, is defined as $\mathbf{ww}':=w_1\cdots w_kw'_1\cdots w'_\ell$. Clearly, concatenation is associative, so we need no parentheses when writing down products of several words. If $\mathbf{w}$ is any word and $\mathbf{w'}$ is empty, then  $\mathbf{ww}':=\mathbf{w}$ and $\mathbf{w}'\mathbf{w}:=\mathbf{w}$. If $\mathbf{w}=\mathbf{pfs}$ for words $\mathbf{p}$, $\mathbf{f}$ and $\mathbf{s}$ (some of which may be empty), we say that $\mathbf{p}$ is a \emph{prefix}, $\mathbf{f}$ is a \emph{factor}, and $\mathbf{s}$ is a \emph{suffix} of the word $\mathbf{w}$.

Let $\mathbf{u}=u_1\cdots u_k$ and $\mathbf{v}$ be words. We say that $\mathbf{u}$ \emph{occurs as a  subword in $\mathbf{v}$} if there exist words $\mathbf{v}_1, \mathbf{v}_2, \dots,  \mathbf{v}_k, \mathbf{v}_{k+1}$ (some of which may be empty) such that
\begin{equation}
\label{eq:occurrence}
\mathbf{v} = \mathbf{v}_1u_1\mathbf{v}_2\cdots \mathbf{v}_ku_k\mathbf{v}_{k+1}.
\end{equation}
Thus, $\mathbf{u}$ as a sequence of variables is a subsequence in $\mathbf{v}$. We refer to any decomposition of the form \eqref{eq:occurrence} as an \emph{occurrence} of $\mathbf{u}$ as a  subword in $\mathbf{v}$. We say that $\mathbf{u}$ occurs in a word $\mathbf{v}$ \emph{with gaps} $G_1,G_2,\dots,G_{k+1}$ if there is an occurrence \eqref{eq:occurrence} of $\mathbf{u}$ as a  subword in $\mathbf{v}$ such that $\alf(\mathbf{v}_\ell)=G_\ell$ for each $\ell=1,2,\dots,k+1$.

\begin{example}
The three occurrences of the word $x$ as a subword in $x^2yx$ are:
\[
x^2yx=\begin{cases}
\underline{x}\cdot xyx,&\text{with gaps $\varnothing,\{x,y\}$;}\\
x\cdot\underline{x}\cdot yx,&\text{with gaps $\{x\},\{x,y\}$};\\
x^2y\cdot\underline{x},&\text{with gaps $\{x,y\},\varnothing$.}
\end{cases}
\]
The two occurrences of the word $xy$ as a  subword in $x^2yx$ are:
\[
x^2yx=\begin{cases}
\underline{x}\cdot x\cdot\underline{y}\cdot x,&\text{with gaps $\varnothing,\{x\},\{x\}$;}\\
x\cdot\underline{x}\underline{y}\cdot x,&\text{with gaps $\{x\},\varnothing,\{x\}$.}
\end{cases}
\]
\end{example}
We adopt the convention that the empty word occurs as a subword in every word $\mathbf{v}$ in a unique way with gap $G_1:=\alf(\mathbf{v})$.

\subsection{Semigroup identities and inequalities}
Let $\mathbf{w}=w_1\cdots w_k$ be a word and $(S,\cdot)$ a semigroup. A \emph{substitution} is any map $\varphi\colon\alf(\mathbf{w})\to S$. The \emph{value} of $\mathbf{w}$ under $\varphi$ is the element $\mathbf{w}\varphi\in S$ defined as $\mathbf{w}\varphi:=w_1\varphi\cdots w_k\varphi$.

A (\emph{semigroup}) \emph{identity} is a formal expression of the form $\mathbf{w}\approx\mathbf{w}'$ where $\mathbf{w}$ and $\mathbf{w}'$ are words. A semigroup $(S,\cdot)$ \emph{satisfies} $\mathbf{w}\approx\mathbf{w}'$ (or $\mathbf{w}\approx\mathbf{w}'$ \emph{holds} in $(S,\cdot)$) if $\mathbf{w}\varphi=\mathbf{w}'\varphi$ for every substitution $\varphi\colon\alf(\mathbf{ww'})\to S$. That is, each substitution of~elements from $S$ for the variables occurring in the words $\mathbf{w}$ or $\mathbf{w}'$ yields equal values to these words.

An \emph{ordered semigroup} is a structure $(S,\cdot,\le)$ such that $(S,\cdot)$ is a semigroup, $(S,\le)$ is a (partially) ordered set and the order is compatible with multiplication in the sense that $a\le b$ implies $ca\le cb$ and $ac\le bc$ for all $a,b,c\in S$. A (\emph{semigroup}) \emph{inequality} is a formal expression of the form $\mathbf{w}\preccurlyeq\mathbf{w}'$ where $\mathbf{w}$ and $\mathbf{w}'$ are words. We refer to $\mathbf{w}$ and $\mathbf{w}'$ as the \emph{lower} and, respectively, \emph{upper} sides of the inequality $\mathbf{w}\preccurlyeq\mathbf{w}'$. An ordered semigroup $(S,\cdot,\le)$ \emph{satisfies} $\mathbf{w}\preccurlyeq\mathbf{w}'$ (or $\mathbf{w}\preccurlyeq\mathbf{w}'$ \emph{holds} in $(S,\cdot,\le)$) if $\mathbf{w}\varphi\le\mathbf{w}'\varphi$ for every substitution $\varphi\colon\alf(\mathbf{ww'})\to S$. Obviously, an identity $\mathbf{w}\approx\mathbf{w}'$ holds in an ordered semigroup $(S,\cdot,\le)$ if and only if $(S,\cdot,\le)$ satisfies both $\mathbf{w}\preccurlyeq\mathbf{w}'$ and $\mathbf{w}'\preccurlyeq\mathbf{w}$.

\subsection{Semiring identities and inequalities}
Most of the above notions naturally extend to \ais{}s. A (\emph{semiring}) \emph{polynomial} is a finite nonempty set of semigroup words. We denote polynomials by uppercase boldface letters and write them as formal sums of their elements. Thus, $\mathbf{W}=\mathbf{w}_1+\dots+\mathbf{w}_p$ means nothing but $\mathbf{W}=\{\mathbf{w}_1,\dots,\mathbf{w}_p\}$; in particular, ordering of summands does not matter.

The above use of the term `polynomial' is quite common in the theory of \ais{}s. To avoid possible confusion, we emphasize that, unlike the `usual' polynomials (which occasionally show up in Subsect.~\ref{subsec:complexity}), semiring polynomials are formed over non-commuting variables, and their monomials (that is, words) carry no coefficients

For a polynomial $\mathbf{W}=\mathbf{w}_1+\dots+\mathbf{w}_p$, we let $\alf(\mathbf{W}):=\cup_{i=1}^p\alf(\mathbf{w}_i)$. Further, we say that a word $\mathbf{u}$ of length $k$ \emph{occurs in $\mathbf{W}$ with gaps} $G_1,G_2,\dots,G_{k+1}$ if $\mathbf{u}$ occurs with gaps $G_1,G_2,\dots,G_{k+1}$ in some of the words $\mathbf{w}_i$, $i=1,\dots,p$.

If $\mathbf{W}=\mathbf{w}_1+\dots+\mathbf{w}_p$ is a polynomial, $(S,+,\cdot)$ is an \ais, and $\varphi\colon\alf(\mathbf{W})\to S$ is any substitution, then the \emph{value} $\mathbf{W}\varphi$ of $\mathbf{W}$ under $\varphi$ is defined as $\sum_{i=1}^p\mathbf{w}_i\varphi$. As the semigroup $(S,+)$ is commutative and idempotent, the value $\mathbf{W}\varphi$ is well-defined because the element $\sum_{i=1}^p\mathbf{w}_i\varphi\in S$ does not depend on how $\mathbf{W}$ is represented as a sum of words.

A (\emph{semiring}) \emph{identity} is a formal expression of the form $\mathbf{W}\approx\mathbf{W}'$ where $\mathbf{W}$ and $\mathbf{W}'$ are polynomials. An \ais{} $(S,+,\cdot)$ \emph{satisfies} $\mathbf{W}\approx\mathbf{W}'$ (or $\mathbf{W}\approx\mathbf{W}'$ \emph{holds} in $(S,+,\cdot)$) if $\mathbf{W}\varphi=\mathbf{W}'\varphi$ for every substitution $\varphi\colon\alf(\mathbf{W})\cup\alf(\mathbf{W}')\to S$.

Given any commutative and idempotent semigroup $(S,+)$, the binary relation $\le$ defined for $a,b\in S$ by
\begin{equation}\label{eq:order}
a\le b\ \text{ if and only if }\ a+b=b
\end{equation}
is known (and easy to see) to be an order on $S$ compatible with addition. If $(S,+,\cdot)$ is an \ais, the distributivity of multiplication over addition on the left and right readily implies that the order \eqref{eq:order} is also compatible with multiplication. When we talk about an order on an \ais, we always mean the order \eqref{eq:order}.

A (\emph{semiring}) \emph{inequality} is a formal expression of the form $\mathbf{W}\preccurlyeq\mathbf{W}'$ where $\mathbf{W}$ and $\mathbf{W}'$ are polynomials referred to as the \emph{lower} and, respectively, \emph{upper} sides of the inequality. An \ais{} $(S,+,\cdot)$ \emph{satisfies} $\mathbf{W}\preccurlyeq\mathbf{W}'$ (or $\mathbf{W}\preccurlyeq\mathbf{W}'$ \emph{holds} in $(S,+,\cdot)$) if for every substitution $\varphi\colon\alf(\mathbf{W})\cup\alf(\mathbf{W}')\to S$, one has $\mathbf{W}\varphi\le\mathbf{W}'\varphi$. Obviously, an identity $\mathbf{W}\approx\mathbf{W}'$ holds in an \ais{} if and only if the \ais{} satisfies both the inequality $\mathbf{W}\preccurlyeq\mathbf{W}'$ and the opposite inequality $\mathbf{W}'\preccurlyeq\mathbf{W}$.

\section{Characterizing inequalities and identities in $(T_n,+,\cdot)$ and $(T_n,\cdot)$}
\label{sec:main}

\subsection{Characterization}
\begin{theorem}
\label{thm:ais-inequality}
The \ais{} \textup[ordered semigroup\textup] of all Boolean upper triangular $n\times n$-matrices satisfies a semiring \textup[respectively, semigroup\textup] inequality if and only if every word of length $k<n$ that occurs in the lower side of the inequality with gaps $G_1,G_2,\dots,G_{k+1}$ occurs in the upper side of the inequality with gaps $G'_1,G'_2,\dots,G'_{k+1}$ such that $G'_\ell\subseteq G_\ell$ for~each $\ell=1,2,\dots,k{+}1$.
\end{theorem}

We start with a lemma yielding a matrix criterion for when a subword occurs in a word with specified gaps. The essence of this lemma is known (see the proof of necessity in \cite[Theorem 2]{PiSt85}); nevertheless, we provide a detailed proof for the sake of completeness.

Denote by $e_{ij}$ the $(i,j)$-th \emph{matrix unit}, i.e., the $n\times n$-matrix with 1 in the position $(i,j)$ and 0 in all other positions. The \emph{zero matrix}, i.e., the $n\times n$-matrix with 0 in all positions, will be denoted simply by 0; this should not cause any confusion.

If $\mathbf{u}$ is a word of length $k$, we denote by $u_\ell$ its variable at its $\ell$-th position from the left, so that $\mathbf{u}=u_1\cdots u_k$ if $k>0$. For each word $\mathbf{v}$ in which $\mathbf{u}$ occurs as a subword with gaps $G_1,G_2,\dots,G_{k+1}$ and each $n>k$, define a substitution $\varphi_{\mathbf{u},\mathbf{v}}\colon\alf(\mathbf{v})\to T_n$ as follows:
\begin{equation}\label{eq:substitution}
x\varphi_{\mathbf{u},\mathbf{v}}:=\sum_{x\in G_\ell} e_{\ell\,\ell}+\sum_{x=u_\ell} e_{\ell\,\ell{+}1}.
\end{equation}
For a variable $x\in\alf(\mathbf{v})$, the rule \eqref{eq:substitution} states that the matrix $x\varphi_{\mathbf{u},\mathbf{v}}$ has 1s only in the main diagonal and the superdiagonal. The superdiagonal entry $\bigl(x\varphi_{\mathbf{u},\mathbf{v}}\bigr)_{\ell\,\ell{+}1}$ is equal to 1 if and only if $x$ appears in the word $\mathbf{u}$ at its $\ell$-th position from the left, while the diagonal entry $\bigl(x\varphi_{\mathbf{u},\mathbf{v}}\bigr)_{\ell\,\ell}$ is equal to~1 if and only $x$ appears in gap $G_\ell$ of the chosen occurrence of $\mathbf{u}$ in $\mathbf{v}$. For instance, if $k=3$ and $n=4$, say, and $x$ appears in the word $\mathbf{v}$ as shown by the arrows in the following scheme:
\[
\mathbf{v} = \stackrel{\stackrel{x}{\downarrow}}{\mathbf{v}}_1 u_1 \mathbf{v}_2 \stackrel{\stackrel{x}{\downarrow}}{u}_2 \mathbf{v}_3\stackrel{\stackrel{x}{\downarrow}}{u}_3\stackrel{\stackrel{x}{\downarrow}}{\mathbf{v}}_4,
\]
then
\[
x\varphi_{\mathbf{u},\mathbf{v}}=e_{11}+e_{23}+e_{34}+e_{44}=\begin{pmatrix}1 & 0 & 0 & 0\\ 0 & 0 & 1 & 0\\ 0 & 0 & 0 & 1\\ 0 & 0 & 0 & 1\end{pmatrix}.
\]
Whenever we need to extend the substitution $\varphi_{\mathbf{u},\mathbf{v}}$ to variables beyond $\alf(\mathbf{v})$, we assume that such variables are mapped by $\varphi_{\mathbf{u},\mathbf{v}}$ to the zero $n\times n$-matrix 0.

\begin{lemma}
\label{lem:substitution}
Let $\mathbf{u}$ be a word of length $k<n$ and let $\mathbf{v}$ be a word in which $\mathbf{u}$ occurs as a subword with gaps $G_1,G_2,\dots,G_{k+1}$. Then for every word $\mathbf{w}$, the word $\mathbf{u}$ occurs in  $\mathbf{w}$ as a subword with gaps $G'_1,G'_2,\dots,G'_{k+1}$ such that $G'_\ell\subseteq G_\ell$ for~each $\ell=1,2,\dots,k{+}1$ if and only if $\bigl(\mathbf{w}\varphi_{\mathbf{u},\mathbf{v}}\bigr)_{1\,k+1}=1$.
\end{lemma}

\begin{proof} Write the word $\mathbf{w}$ as a product of variables, $\mathbf{w}=w_1\cdots w_m$ and, for each $s=1,\dots,m$, let $w_s\varphi_{\mathbf{u},\mathbf{v}}:=\bigl(\alpha_{ij}^{(s)}\bigr)_{n\times n}$. Since all matrices $\bigl(\alpha_{ij}^{(s)}\bigr)_{n\times n}$ are upper triangular, we get the following expansion of the entry of the matrix $\mathbf{w}\varphi_{\mathbf{u},\mathbf{v}}$ in the position $(1,k+1)$:
\begin{equation}
\label{eq:expansion}
\bigl(\mathbf{w}\varphi_{\mathbf{u},\mathbf{v}}\bigr)_{1\,k+1}=\sum_{1\le j_1\le\dots\le j_{m-1}\le k+1}\alpha_{1j_1}^{(1)}\alpha_{j_1j_2}^{(2)}\cdots\alpha_{j_{m-1}k+1}^{(m)}.
\end{equation}

\smallskip

\noindent\emph{Necessity.} Suppose that $\mathbf{u}$ occurs in $\mathbf{w}$ as a subword with gaps $G'_1,G'_2,\dots,G'_{k+1}$ such that $G'_\ell\subseteq G_\ell$ for~each $\ell=1,2,\dots,k{+}1$.

If $k=0$ (that is, the word $\mathbf{u}$ is empty), then the sum  in the right-hand side of~\eqref{eq:expansion} reduces to the product $\alpha_{11}^{(1)}\alpha_{11}^{(2)}\cdots\alpha_{11}^{(m)}$. Recall that, by our convention, the empty word occurs as a subword in every word with gap equal to the alphabet of this word. Thus, we have $G_1=\alf(\mathbf{v})$ and $G'_1=\alf(\mathbf{w})$. By the rule \eqref{eq:substitution}, $x\varphi_{\mathbf{u},\mathbf{v}}=e_{11}$ for all $x\in G_1$. As $G'_1\subseteq G_1$, we have $x\varphi_{\mathbf{u},\mathbf{v}}=e_{11}$ for all $x\in G'_1$ whence $\alpha_{11}^{(s)}=1$ for each $s=1,\dots,m$. Therefore, $\bigl(\mathbf{w}\varphi_{\mathbf{u},\mathbf{v}}\bigr)_{11}=\alpha_{11}^{(1)}\alpha_{11}^{(2)}\cdots\alpha_{11}^{(m)}=1$.

Now assume $k>0$. Since $\mathbf{u}=u_1\cdots u_k$ occurs in $\mathbf{w}$ with gaps $G'_1,G'_2,\dots,G'_{k+1}$ such that $G'_\ell\subseteq G_\ell$ for~each $\ell=1,2,\dots,k{+}1$, there exists a sequence $1\le s_1<\dots<s_k\le m$ such that $u_\ell=w_{s_\ell}$ for each $\ell=1,\dots,k$, and moreover, if we let $s_0:=0$ and $s_{k+1}:=m+1$, then $\alf(w_{s_{\ell-1}+1}\cdots w_{s_\ell-1})=G'_\ell\subseteq G_\ell$ for each $\ell=1,\dots,k+1$. Then the rule \eqref{eq:substitution} implies that all factors in the product
\begin{multline*}
\alpha_{11}^{(1)}\cdots\alpha_{11}^{(s_1-1)}\alpha_{12}^{(s_1)}\alpha_{22}^{(s_1+1)}\cdots\alpha_{22}^{(s_2-1)}\alpha_{23}^{(s_2)}\cdots\\
\cdots\alpha_{k\,k}^{(s_{k-1}+1)}\cdots\alpha_{k\,k}^{(s_k-1)}\alpha_{k\,k+1}^{(s_k)}\alpha_{k+1\,k+1}^{(s_k+1)}\cdots\alpha_{k+1\,k+1}^{(m)}
\end{multline*}
are equal to 1, whence the product itself equals 1 and so does the sum in the right-hand side of~\eqref{eq:expansion} of which this product is one of the summands. Thus, $\bigl(\mathbf{w}\varphi_{\mathbf{u},\mathbf{v}}\bigr)_{1\,k+1}=1$.

\medskip

\noindent\emph{Sufficiency.} Suppose that $\bigl(\mathbf{w}\varphi_{\mathbf{u},\mathbf{v}}\bigr)_{1\,k+1}=1$, that is, the sum in the expansion \eqref{eq:expansion} equals 1.  Then the sum must have a summand equal to 1. Fix some indices $j_1,\dots,j_{m-1}$ such that
\begin{equation}
\label{eq:product}
\alpha_{1j_1}^{(1)}\alpha_{j_1j_2}^{(2)}\cdots\alpha_{j_{m-1}\ k+1}^{(m)}=1.
\end{equation}
Recall that we set $x\varphi_{\mathbf{u},\mathbf{v}}=0$ for all $x\notin\alf(\mathbf{v})$. From \eqref{eq:product}, we see that $x\varphi_{\mathbf{u},\mathbf{v}}\ne0$ for all $x\in\alf(\mathbf{w})$. We conclude that $\alf(\mathbf{w})\subseteq\alf(\mathbf{v})$. This proves our claim for the case $k=0$, that is, for $\mathbf{u}$ being empty. Indeed, by the convention we adopted, the empty word occurs as a  subword in $\mathbf{v}$ with gap $G_1=\alf(\mathbf{v})$ and in $\mathbf{w}$ with gap $G'_1=\alf(\mathbf{w})$ whence $G'_1\subseteq G_1$ as required.

So, assume $k>0$. As the matrices $x\varphi_{\mathbf{u},\mathbf{v}}$ can have 1s in the main diagonal and the superdiagonal only, $\alpha_{ij}^{(t)}=0$ whenever $j\ne i,i+1$. Therefore, for the equality \eqref{eq:product} to hold, all strict inequalities in the sequence $j_0:=1\le j_1\le\dots\le j_{m-1}\le k+1 =:j_{m}$ must be of the form $j_r<j_{r+1}=j_r+1$. To ascend from $j_0=1$ to $j_m = k+1$, exactly $k$ steps of the form $j_{r+1} = j_r + 1$ are required. For $\ell = 1, \dots, k$, let $t_\ell$ denote the position of the $\ell$-th such step from the left. We see that the product $\alpha_{1j_1}^{(1)}\alpha_{j_1j_2}^{(2)}\cdots\alpha_{j_{m-1}k+1}^{(m)}$ in \eqref{eq:product} must have the form
\begin{multline}
\label{eq:longproduct}
\alpha_{11}^{(1)}\cdots\alpha_{11}^{(t_1-1)}\alpha_{12}^{(t_1)}\alpha_{22}^{(t_1+1)}\cdots\alpha_{22}^{(t_2-1)}\alpha_{23}^{(t_2)}\cdots\\
\cdots\alpha_{k\,k}^{(t_{k-1}+1)}\cdots\alpha_{k\,k}^{(t_k-1)}\alpha_{k\,k+1}^{(t_k)}\alpha_{k+1\,k+1}^{(t_k+1)}\cdots\alpha_{k+1\,k+1}^{(m)},
\end{multline}
with all factors equal to 1. By the rule~\eqref{eq:substitution}, the equality $\alpha_{\ell\,\ell+1}^{(t_\ell)}=1$ is only possible if $w_{t_\ell}=u_\ell$. This fixes an occurrence of the word $\mathbf{u}=u_1\cdots u_k$ as a subword of the word $\mathbf{w}$. For convenience, let $t_0:=0$ and $t_{k+1}:=m{+}1$. Now for each $\ell=1,\dots,k{+}1$, consider the segment $\alpha_{\ell\,\ell}^{(t_{\ell-1}+1)}\cdots\alpha_{\ell\,\ell}^{(t_\ell-1)}$ of the product \eqref{eq:longproduct}, that is, the $\ell$-th from the left maximal segment of \eqref{eq:longproduct} whose entries come from the main diagonals of the matrices $w_t\varphi_{\mathbf{u},\mathbf{v}}$. The segment $\alpha_{\ell\,\ell}^{(t_{\ell-1}+1)}\cdots\alpha_{\ell\,\ell}^{(t_\ell-1)}$ corresponds to the factor $w_{t_{\ell-1}+1}\cdots w_{t_\ell-1}$ of the word $\mathbf{w}$. If $G'_\ell:=\alf(w_{t_{\ell-1}+1}\cdots w_{t_\ell-1})$, then the sets $G'_1,G'_2,\dots,G'_{k+1}$ constitute gaps corresponding to the occurrence of $\mathbf{u}$ in $\mathbf{w}$ that we just fixed. We have $\alpha_{\ell\,\ell}^{(t)}=1$ for all $t_{\ell-1}<t<t_\ell$ but by the rule~\eqref{eq:substitution}, this is only possible if all variables $w_{t_{\ell-1}+1},\dots,w_{t_\ell-1}$ belong to the set $G_\ell$. Hence, $G'_\ell\subseteq G_\ell$ for each $\ell=1,\dots,k+1$.
\end{proof}

\begin{proof}[Proof of Theorem~\ref{thm:ais-inequality}]
Semigroup inequalities are special instances of semiring ones, so we consider only the latter in the proof.

\medskip

\noindent\emph{Necessity.} Take an inequality $\mathbf{W}\preccurlyeq\mathbf{W}'$ holding in the \ais{} $(T_n,+,\cdot)$ of all Boolean upper triangular $n\times n$-matrices and consider an arbitrary word $\mathbf{u}$ of length $k<n$ that occurs in $\mathbf{W}$ with some gaps $G_1,G_2,\dots,G_{k+1}$. This means that $\mathbf{u}$ occurs with these gaps in some summand of the polynomial $\mathbf{W}$. We fix such a summand $\mathbf{w}$ and an occurrence of $\mathbf{u}$ in the word $\mathbf{w}$ with gaps $G_1,G_2,\dots,G_{k+1}$.

By Lemma~\ref{lem:substitution} (with $\mathbf{w}$ in the role of the word $\mathbf{v}$), we have $\bigl(\mathbf{w}\varphi_{\mathbf{u},\mathbf{w}}\bigr)_{1\,k+1}=1$. Hence, $\bigl(\mathbf{W}\varphi_{\mathbf{u},\mathbf{w}}\bigr)_{1\,k+1}=1$ since the matrix $\mathbf{w}\varphi_{\mathbf{u},\mathbf{w}}$ is one of the summands of the matrix $\mathbf{W}\varphi_{\mathbf{u},\mathbf{w}}$.

Since inequality $\mathbf{W}\preccurlyeq\mathbf{W}'$ holds in the \ais{} $(T_n,+,\cdot)$, we have $\mathbf{W}\varphi_{\mathbf{u},\mathbf{w}}\le\mathbf{W'}\varphi_{\mathbf{u},\mathbf{w}}$. The meaning of the order \eqref{eq:order} for Boolean matrices is quite transparent: $A\le B$ if and only if the matrix $B$ has 1 in every position in which the matrix $A$ has 1. We conclude that $\bigl(\mathbf{W'}\varphi_{\mathbf{u},\mathbf{w}}\bigr)_{1\,k+1}=1$, whence $\bigl(\mathbf{w'}\varphi_{\mathbf{u},\mathbf{w}}\bigr)_{1\,k+1}=1$ for a summand $\mathbf{w'}$ of the polynomial $\mathbf{W'}$. By Lemma~\ref{lem:substitution}, $\mathbf{u}$ occurs in  $\mathbf{w'}$ as a subword with gaps $G'_1,G'_2,\dots,G'_{k+1}$ such that $G'_\ell\subseteq G_\ell$.

\medskip

\noindent\emph{Sufficiency.} Take two polynomials $\mathbf{W}$ and $\mathbf{W'}$ with the property that every word of length $k<n$ that occurs in $\mathbf{W}$ with gaps $G_1,G_2,\dots,G_{k+1}$ occurs in $\mathbf{W'}$ with gaps $G'_1,G'_2,\dots,G'_{k+1}$ such that $G'_\ell\subseteq G_\ell$ for each $\ell=1,2,\dots,k+1$. We aim to show that the \ais{} $(T_n,+,\cdot)$ of all Boolean upper triangular $n\times n$-matrices satisfies the inequality $\mathbf{W}\preccurlyeq\mathbf{W}'$, that is,  $\mathbf{W}\varphi\le\mathbf{W'}\varphi$ for an arbitrary substitution $\varphi\colon\alf(\mathbf{W})\cup\alf(\mathbf{W}')\to T_n$. According to the  meaning of the order \eqref{eq:order} for Boolean matrices, we need to show that $\bigl(\mathbf{W'}\varphi\bigr)_{p\,q}=1$ whenever $\bigl(\mathbf{W}\varphi\bigr)_{p\,q}=1$.

Fix a pair of indices $(p,q)$ such that $\bigl(\mathbf{W}\varphi\bigr)_{p\,q}=1$. There must be a summand $\mathbf{w}$ of the polynomial $\mathbf{W}$ such that $\bigl(\mathbf{w}\varphi\bigr)_{p\,q}=1$; we fix such a summand.
The matrix $\mathbf{w}\varphi$ is upper triangular, so we have $p\le q$.

First consider the case where $q=p$. Write the word $\mathbf{w}$ as a product of variables, $\mathbf{w}=w_1\cdots w_m$, and let $w_s\varphi:=\bigl(\alpha_{ij}^{(s)}\bigr)_{n\times n}$ for each $s=1,\dots,m$. Clearly, for each $p=1,\dots,n$,
\[
\bigl(\mathbf{w}\varphi\bigr)_{p\,p}=\prod_{1\le s\le m}\alpha_{p\,p}^{(s)},
\]
and therefore, $\bigl(\mathbf{w}\varphi\bigr)_{p\,p}=1$ implies $\bigl(x\varphi\bigr)_{p\,p}=1$ for each variable $x\in\alf(\mathbf{w})$. By our convention, the empty word occurs as a  subword in $\mathbf{w}$ with gap $G_1=\alf(\mathbf{w})$. Then it must occur in some summand $\mathbf{w'}$ of the polynomial $\mathbf{W}'$ with gap $G'_1=\alf(\mathbf{w}')$ such that $G'_1\subseteq G_1$. Hence $\alf(\mathbf{w}')\subseteq\alf(\mathbf{w})$ which ensures $\bigl(\mathbf{w'}\varphi\bigr)_{p\,p}=1$.

Thus, we may assume that $p<q$ for the rest of the proof. Then
\[
\bigl(\mathbf{w}\varphi\bigr)_{p\,q}=\sum_{p\le j_1\le\dots\le j_{m-1}\le q} \alpha_{p\,j_1}^{(1)}\alpha_{j_1j_2}^{(2)}\cdots\alpha_{j_{m-1}\,q}^{(m)},
\]
and since the sum on the right-hand side is equal to 1, so is one of its summands. Let the product $\alpha_{p\,j_1}^{(1)}\alpha_{j_1j_2}^{(2)}\cdots\alpha_{j_{m-1}\,q}^{(m)}$ be such a summand. For convenience of notation, let $j_0:=p$ and $j_m:=q$. Since $p<q$, among the inequalities $j_0\le j_1\le\dots\le j_{m-1}\le j_m$, there must be some strict ones. Let $j_{s_1-1}< j_{s_1}$,  $j_{s_2-1}< j_{s_2}$, \dots, $j_{s_k-1}< j_{s_k}$ be all such strict inequalities and let $d_\ell:=j_{s_\ell}-j_{s_\ell-1}$ for each $\ell=1,\dots,k$. Then the chosen product can be written as
\begin{multline}
\label{eq:longproduct1}
\alpha_{p\,p}^{(1)}\cdots\alpha_{p\,p}^{(s_1-1)}\alpha_{p\ p+d_1}^{(s_1)}\alpha_{p+d_1\,p+d_1}^{(s_1+1)}\cdots\alpha_{p+d_1\,p+d_1}^{(s_2-1)}\alpha_{p+d_1\,p+d_1+d_2}^{(s_2)}\cdots\\
\cdots\alpha_{q-d_k\ q-d_k}^{(s_{k-1}+1)}\cdots\alpha_{q-d_k\ q-d_k}^{(s_{k-1}-1)}\alpha_{q-d_k\ q}^{(s_k)}\alpha_{q\,q}^{(s_k+1)}\cdots\alpha_{q\,q}^{(m)}.
\end{multline}
Since the product is equal to 1, so is each of its factors. In particular,
\begin{equation}
\label{eq:conclusion}
\alpha_{p\ p+d_1}^{(s_1)}=\alpha_{p+d_1\,p+d_1+d_2}^{(s_2)}=\dots=\alpha_{q-d_k\ q}^{(s_k)}=1.
\end{equation}
Consider the  subword $\mathbf{u}:=w_{s_1}w_{s_2}\cdots w_{s_k}$ of the word $\mathbf{w}$. The length $k$ of $\mathbf{u}$ is the number of strict inequalities among $j_0\le j_1\le\dots\le j_{m-1}\le j_m$; hence $k\le j_m-j_0=q-p\le n-1$. The word $\mathbf{u}$ occurs in $\mathbf{w}$ with gaps $G_\ell:=\alf(\prod_{s_{\ell-1}<s<s_\ell}w_s)$, where $\ell=1,2,\dots,k+1$ and we set $s_0:=0$ and $s_{k+1}:=m+1$ so that for $\ell=1$ and $\ell=m+1$, the expression $\prod_{s_{\ell-1}<s<s_\ell}w_s$ represents the prefix $w_1\cdots w_{s_1-1}$ and, respectively, the suffix $w_{s_k+1}\cdots w_m$ of the word $\mathbf{w}$. By the condition, $\mathbf{u}$ must occur in some summand of the polynomial $\mathbf{W}'$ with gaps $G'_1,G'_2,\dots,G'_{k+1}$ such that $G'_\ell\subseteq G_\ell$ for each $\ell=1,2,\dots,k+1$. Abusing notation, we denote this summand by $\mathbf{w'}$ even though it need not be the same as the one used in the above analysis of the case where $q=p$.

Fix an occurrence of $\mathbf{u}$ as a  subword in $\mathbf{w'}$ with $G'_1,G'_2,\dots,G'_{k+1}$. Write the word $\mathbf{w'}$ as a product of variables, $\mathbf{w'}=w'_1\cdots w'_{m'}$, and let $t_1,t_2,\dots,t_k$ be the positions corresponding to the chosen occurrence of $\mathbf{u}$ in $\mathbf{w'}$. This means that, first, $w'_{t_\ell}=w_{s_\ell}$ for each $\ell=1,2,\dots,k$, and second, $G'_\ell:=\alf(\prod_{t_{\ell-1}<t<t_\ell}w'_t)\subseteq G_\ell$, where $\ell=1,2,\dots,k+1$ and, similarly to the above, we set $t_0:=0$ and $t_{k+1}:=m'+1$ in order to accommodate all gaps in a uniform expression, including the extremes $G'_1$ and $G'_{k+1}$. In the below scheme, the words $\mathbf{w}$ and $\mathbf{w'}$ are aligned at their common  subword $\mathbf{u}$, with both gap sequences explicitly written down.
\begin{center}
\begin{tabular}{rcccccl}
$\overbrace{w_1\cdots w_{s_1-1}}^{G_1=\alf(w_1\cdots w_{s_1-1})}\cdot$ & $w_{s_1}$ & $\cdot\overbrace{w_{s_1+1}\cdots w_{s_2-1}}^{G_2=\alf(w_{s_1+1}\cdots w_{s_2-1})}\cdot$ & $w_{s_2}$ & \dots & $w_{s_k}$ & $\cdot\overbrace{w_{s_k+1}\cdots w_m}^{G_{k+1}=\alf(w_{s_k+1}\cdots w_m)}$ \\
                     & $\shortparallel\rule{3pt}{0pt}$ &                       & $\shortparallel\rule{3pt}{0pt}$ & \dots & $\shortparallel\rule{3pt}{0pt}$ &  \\
$\underbrace{w'_1\cdots w'_{t_1-1}}_{G_1'=\alf(w'_1\cdots w'_{t_1-1})}\cdot$ & $w'_{t_1}$ & $\cdot\underbrace{w'_{t_1+1}\cdots w'_{t_2-1}}_{G_2'=\alf(w'_{t_1+1}\cdots w'_{t_2-1})}\cdot$ & $w'_{t_2}$ & \dots & $w'_{t_k}$ & $\cdot\underbrace{w'_{t_k+1}\cdots w'_{m'}}_{G'_{k+1}=\alf(w'_{t_k+1}\cdots w'_{m'})}$
\end{tabular}
\end{center}

Let $w'_t\varphi:=\bigl(\beta_{ij}^{(t)}\bigr)_{n\times n}$ for each $t=1,\dots,m'$ and consider the product
\begin{multline}
\label{eq:longproduct2}
\beta_{p\,p}^{(1)}\cdots\beta_{p\,p}^{(t_1-1)}\beta_{p\ p+d_1}^{(t_1)}\beta_{p+d_1\,p+d_1}^{(t_1+1)}\cdots\beta_{p+d_1\,p+d_1}^{(t_2-1)}\beta_{p+d_1\,p+d_1+d_2}^{(t_2)}\cdots\\
\cdots\beta_{q-d_k\ q-d_k}^{(t_{k-1}+1)}\cdots\beta_{q-d_k\ q-d_k}^{(t_{k-1}-1)}\beta_{q-d_k\ q}^{(t_k)}\beta_{q\,q}^{(t_k+1)}\cdots\beta_{q\,q}^{(m')},
\end{multline}
that appears as one of the summands in the expansion of $\bigl(\mathbf{w'}\varphi\bigr)_{p\,q}$ via entries of the matrices $\bigl(\beta_{ij}^{(t)}\bigr)_{n\times n}$, $t=1,\dots,m'$. We aim to show that the product \eqref{eq:longproduct2} is equal to 1, which will yield the equality $\bigl(\mathbf{w'}\varphi\bigr)_{p\,q}=1$.

The equalities $w'_{t_\ell}=w_{s_\ell}$ imply the matrix equalities $\bigl(\beta_{ij}^{(t_\ell)}\bigr)_{n\times n}=\bigl(\alpha_{ij}^{(s_\ell)}\bigr)_{n\times n}$ for each $\ell=1,2,\dots,k$. Hence, from \eqref{eq:conclusion} we get
\[
\beta_{p\ p+d_1}^{(t_1)}=\beta_{p+d_1\,p+d_1+d_2}^{(t_2)}=\dots=\beta_{q-d_k\ q}^{(t_k)}=1,
\]
that is, all factors of the product \eqref{eq:longproduct2} that are not on the main diagonals of the matrices involved in the product are equal to 1. It remains to understand why the factors from the main diagonals are also equal to 1.

From the fact that the product \eqref{eq:longproduct1} is equal to 1, we conclude that if $x$ is a variable belonging to gap $G_\ell$, then $\bigl(x\varphi)_{j\,j}=1$ for every $j_{s_{\ell-1}}\le j\le j_{s_{\ell}-1}$ and $\ell=1,2,\dots,k{+}1$. (Recall that we have set $s_0=0$ and $s_{k+1}=m+1$ while $j_0=p$ and $j_m=q$ so that $j_{s_0}=j_0=p$ and $j_{s_{k+1}-1}=j_m=q$.) Since $G'_\ell\subseteq G_\ell$ for each $\ell=1,2,\dots,k+1$ and $G'_\ell=\alf(\prod_{t_{\ell-1}<t<t_\ell}w'_t)$, we see that every variable $w'_t$ with $t_{\ell-1}<t<t_\ell$ belongs to $G_\ell$. Hence $\bigl(w'_t\varphi)_{j\,j}=\beta^{t}_{j\,j}=1$ for all $t_{\ell-1}<t<t_\ell$ and $j_{t_{\ell-1}}\le j\le j_{t_{\ell}-1}$. In the notation used in \eqref{eq:longproduct2}, this means
\begin{align*}
  &\beta_{p\,p}^{(t)}=1 && \text{for}\ t=1,\dots,t_1-1,\\
  &\beta_{p+d_1\,p+d_1}^{(t)}=1 && \text{for}\ t=t_1+1,\dots,t_2-1,\\
  &\makebox[2.1cm]{\dotfill} &&\\
  &\beta_{q\,q}^{(t)}=1 && \text{for}\ t=t_k+1,\dots,m'.
\end{align*}
Thus, all factors in the product \eqref{eq:longproduct2} are equal to 1 whence $\bigl(\mathbf{w'}\varphi\bigr)_{p\,q}=1$. This implies that $\bigl(\mathbf{W'}\varphi\bigr)_{p\,q}=1$ as required.
\end{proof}

A characterization of identities holding in the \ais{} $(T_n,+,\cdot)$ or the semigroup $(T_n,\cdot)$ follows immediately.

\begin{corollary}
\label{cor:ais-identity}
The \ais{} \textup[semigroup\textup] of all Boolean upper triangular $n\times n$-matrices satisfies a semiring \textup[respectively, semigroup\textup] identity if and only if every word of length $k<n$ that occurs in one side of the identity with gaps $G_1,G_2,\dots,G_{k+1}$ occurs in the other side of the identity with gaps $G'_1,G'_2,\dots,G'_{k+1}$ such that $G'_\ell\subseteq G_\ell$ \ for each $\ell=1,2,\dots,k{+}1$.
\end{corollary}

We conclude the subsection with a concrete application of Corollary~\ref{cor:ais-identity}. In addition to illustrating our techniques, the identities established in the next example are used in~\cite{Vo:embedding}.

\begin{example}
\label{ex:two identities}
(i) The identity
\begin{equation}
\label{eq:semigroup identity in Tn}x^n\approx x^{n+1},
\end{equation}
holds in the semigroup $(T_n,\cdot)$.

(ii) The identity
\begin{equation}
\label{eq:semiring identity in Tn} x^{n-1}y^{n-1}\approx x^ny^{n-1}+x^{n-1}y^n
\end{equation}
holds in the \ais{} $(T_n,+,\cdot)$.
\end{example}

\begin{proof}
(i) Take any subword $\mathbf{u}=u_1\cdots u_k$ of $x^n$ with $k<n$, and consider its occurrence
\begin{equation}\label{eq:decompositionx}
x^n=\mathbf{v}_1u_1\mathbf{v}_2\cdots\mathbf{v}_ku_k\mathbf{v}_{k+1}.
\end{equation}
Here $u_1=\dots=u_k=x$ and each of the words $\mathbf{v}_\ell$, $\ell=1,2,\dots,k{+}1$, is either empty or a power of $x$. Hence each gap in \eqref{eq:decompositionx} is equal to either $\varnothing$ or $\{x\}$. Since $k<n$, not all of the words $\mathbf{v}_\ell$ are empty. If $i$ is such that the factor $\mathbf{v}_i$ is not empty, consider the decomposition of the word $x^{n+1}$ resulting by inserting $x$ after the factor $\mathbf{v}_i$ in the right-hand side of \eqref{eq:decompositionx}:
\[
x^{n+1}=\mathbf{v}_1u_1\mathbf{v}_2\cdots u_{i-1}\mathbf{v}_ixu_i\cdots\mathbf{v}_ku_k\mathbf{v}_{k+1}.
\]
This is an occurrence of the subword $\mathbf{u}$ in $x^{n+1}$ with the same gaps as in \eqref{eq:decompositionx}.

Conversely, any occurrence
\begin{equation}\label{eq:decompositionxx}
x^{n+1}=\mathbf{w}_1u_1\mathbf{w}_2\cdots\mathbf{w}_ku_k\mathbf{w}_{k+1}
\end{equation}
of $\mathbf{u}$ in $x^{n+1}$ has some nonempty factor $\mathbf{w}_i$. Then the decomposition
\[
x^n=\mathbf{w}_1u_1\mathbf{w}_2\cdots u_{i-1}\mathbf{w}'_iu_i\cdots\mathbf{w}_ku_k\mathbf{w}_{k+1},
\]
where $\mathbf{w}'_i$ is obtained from $\mathbf{w}_i$ by deleting a single instance of $x$, constitutes an occurrence of $\mathbf{u}$ in $x^n$ with gaps contained in the corresponding gaps of \eqref{eq:decompositionxx}.

We have thus verified the identity \eqref{eq:semigroup identity in Tn} fulfils the criterion of Corollary~\ref{cor:ais-identity}.

\medskip

(ii) Take any subword $\mathbf{u}=u_1\cdots u_k$ of $x^{n-1}y^{n-1}$ with $k<n$. Clearly, $\mathbf{u}=x^{p}y^{k-p}$ for some $0\le p\le k$. Any occurrence of $\mathbf{u}$ in $x^{n-1}y^{n-1}$ can be written as
\begin{equation}\label{eq:decompositionxy}
x^{n-1}y^{n-1}=\overbrace{\mathbf{v}_1u_1\mathbf{v}_2\cdots\mathbf{v}_pu_p\mathbf{v}'_{p+1}}^{x^{n-1}}\cdot\overbrace{\mathbf{v}''_{p+1}u_{p+1}\mathbf{v}_{p+2}\cdots\mathbf{v}_ku_k\mathbf{v}_{k+1}}^{y^{n-1}},
\end{equation}
where $u_1=\dots=u_p=x$, $u_{p+1}=\dots=u_k=y$, the words $\mathbf{v}_1$, \dots, $\mathbf{v}_p$, $\mathbf{v}'_{p+1}$ are either empty or powers of $x$, and the words $\mathbf{v}''_{p+1}$, $\mathbf{v}_{p+2}$, \dots, $\mathbf{v}_{k+1}$ are either empty or powers of $y$.

If $0<p<k$ or $k<n-1$, then there is a nonempty factor amongst $\mathbf{v}_1$, \dots, $\mathbf{v}_p$, $\mathbf{v}'_{p+1}$, and inserting $x$ after this factor in the right-hand side of \eqref{eq:decompositionxy}, we obtain a decomposition of the word $x^ny^{n-1}$, which is an occurrence of $\mathbf{u}$ in $x^ny^{n-1}$ with the same gaps as in \eqref{eq:decompositionxy}.

If $p=k=n-1$, then $\mathbf{u}=x^{n-1}$, and this subword occurs in $x^{n-1}y^{n-1}$ and $x^{n-1}y^n$ with the same gaps $\underbrace{\varnothing,\dots,\varnothing}_{\text{$n-1$ times}},\{y\}$.

If $p=0$, $k=n-1$, then $\mathbf{u}=y^{n-1}$, and this subword occurs in $x^{n-1}y^{n-1}$ and $x^ny^{n-1}$ with the same gaps $\{x\},\underbrace{\varnothing,\dots,\varnothing}_{\text{$n-1$ times}}$.

Conversely, take any word $\mathbf{u}=u_1\cdots u_k$ with $k<n$ that occurs as a subword in the polynomial $x^ny^{n-1}+x^{n-1}y^n$. By definition, this means that $\mathbf{u}$ is a subword in one of the words $x^ny^{n-1}$ or $x^{n-1}y^n$. Because of symmetry, it suffices to consider one case; so, assume that $\mathbf{u}$ is a subword of $x^ny^{n-1}$. We have $\mathbf{u}=x^{p}y^{k-p}$ for some $0\le p\le k$; therefore, any occurrence of $\mathbf{u}$ in $x^ny^{n-1}$ can be written as
\begin{equation}\label{eq:decompositionxxy}
x^ny^{n-1}=\overbrace{\mathbf{w}_1u_1\mathbf{w}_2\cdots\mathbf{w}_pu_p\mathbf{w}'_{p+1}}^{x^n}\cdot\overbrace{\mathbf{w}''_{p+1}u_{p+1}\mathbf{w}_{p+2}\cdots\mathbf{w}_ku_k\mathbf{w}_{k+1}}^{y^{n-1}},
\end{equation}
where $u_1=\dots=u_p=x$, $u_{p+1}=\dots=u_k=y$, the words $\mathbf{w}_1$, \dots, $\mathbf{w}_p$, $\mathbf{w}'_{p+1}$ are either empty or powers of $x$, and the words $\mathbf{w}''_{p+1}$, $\mathbf{w}_{p+2}$, \dots, $\mathbf{w}_{k+1}$ are either empty or powers of $y$. Since $p\le k<n$, there is a nonempty factor amongst $\mathbf{w}_1$, \dots, $\mathbf{w}_p$, $\mathbf{w}'_{p+1}$. Deleting a single instance of $x$ from this factor yields  a decomposition of the word $x^{n-1}y^{n-1}$, which is an occurrence of $\mathbf{u}$ in $x^{n-1}y^{n-1}$ with gaps contained in the corresponding gaps of \eqref{eq:decompositionxxy}.

We have thus verified the identity \eqref{eq:semiring identity in Tn} fulfils the criterion of Corollary~\ref{cor:ais-identity}.
\end{proof}

\subsection{Using inclusions is unavoidable}
While inclusions between gaps seem natural for characterizing inequalities of the \ais{} $(T_n,+,\cdot)$ and the ordered semigroup $(T_n,\cdot,\le)$, one might expect that identities $(T_n,+,\cdot)$ and $(T_n,\cdot)$ could be characterized in terms of equalities between gaps rather than inclusions. This, however, is not the case, as we discuss now. We restrict to semigroup identities for simplicity but the situation with semiring identities is the same.

Corollary~\ref{cor:ais-identity} easily implies that every identity $\mathbf{w}\approx\mathbf{w'}$ holding in the semigroup of all Boolean upper triangular $n\times n$-matrices satisfies the following condition $\exists\mathrm{EG}_n$ (existence of equal gaps):
\begin{itemize}
  \item[$\exists\mathrm{EG}_n$:] every common subword of $\mathbf{w}$ and $\mathbf{w'}$ of length $k<n$ has occurrences in $\mathbf{w}$ and $\mathbf{w'}$ with the same gaps.
\end{itemize}
Moreover, we can specify occurrences with equal gaps for each common subword of length $k<n$. An occurrence
\[
\mathbf{v} = \mathbf{v}_1u_1\mathbf{v}_2\cdots \mathbf{v}_ku_k\mathbf{v}_{k+1}
\]
of $\mathbf{u}=u_1\cdots u_k$ in $\mathbf{v}$ is said to be \emph{leftmost} if $u_i\notin\alf(\mathbf{v}_i)$ for all $i=1,\dots,k$. The leftmost occurrence of $\mathbf{u}$ in $\mathbf{v}$ is unique; that is, its factors
$\mathbf{v}_1,\mathbf{v}_2,\dots,\mathbf{v}_k,\mathbf{v}_{k+1}$ are uniquely determined. Indeed, $\mathbf{v}_1$ is uniquely determined as the longest prefix of the word $\mathbf{v}$ in which the variable $u_1$ does not occur. Assuming that $\mathbf{v}_i$ has already been uniquely determined for all $i<k$, the factor $\mathbf{v}_{i+1}$ is uniquely determined as the longest prefix not containing the variable $u_{i+1}$ of the word obtained from $\mathbf{v}$ by removing the prefix $\mathbf{v}_1u_1\mathbf{v}_2\cdots \mathbf{v}_iu_i$. Finally, $\mathbf{v}_{k+1}$ is uniquely determined as the suffix of $\mathbf{v}$ following $\mathbf{v}_1u_1\mathbf{v}_2\cdots \mathbf{v}_ku_k$.

\begin{corollary}
\label{cor:leftmost}
If an identity $\mathbf{w}\approx\mathbf{w'}$ holds in the semigroup $(T_n,\cdot)$, then in $\mathbf{w}$ and $\mathbf{w'}$, the leftmost occurrences of each common subword of length $k<n$ have the same gaps.
\end{corollary}

\begin{proof}
Let a word $\mathbf{u}=u_1\cdots u_k$ with $k<n$ occur in $\mathbf{w}$, and let
\begin{equation}\label{eq:leftmost1}
\mathbf{w} = \mathbf{w}_1u_1\mathbf{w}_2\cdots \mathbf{w}_ku_k\mathbf{w}_{k+1}
\end{equation}
be its leftmost occurrence in $\mathbf{w}$. For $\ell=1,\dots,k+1$, denote $\alf(\mathbf{w}_\ell)$ by $G_\ell$.  By Corollary~\ref{cor:ais-identity}, $\mathbf{u}$ must occur in $\mathbf{w'}$ with gaps $G'_1,G'_2,\dots,G'_{k+1}$ such that $G'_\ell\subseteq G_\ell$ for each $\ell=1,2,\dots,k+1$. Let
\begin{equation}\label{eq:leftmost2}
\mathbf{w'} = \mathbf{w'}_1u_1\mathbf{w'}_2\cdots \mathbf{w'}_ku_k\mathbf{w'}_{k+1}
\end{equation}
be an occurrence of $\mathbf{u}$ in $\mathbf{w'}$ such that $\alf(\mathbf{w'}_\ell)=G'_\ell$ for each $\ell=1,\dots,k+1$. Then $u_i\notin\alf(\mathbf{w'}_i)$ for all $i=1,\dots,k$, because $\alf(\mathbf{w'}_i)=G'_i\subseteq G_i=\alf(\mathbf{w}_i)$ and $u_i\notin\alf(\mathbf{w}_i)$ as \eqref{eq:leftmost1} is the leftmost occurrence of $\mathbf{u}$ in $\mathbf{w}$. We conclude that \eqref{eq:leftmost2} is the leftmost occurrence of $\mathbf{u}$ in $\mathbf{w'}$. Applying Corollary~\ref{cor:ais-identity} to \eqref{eq:leftmost2}, we obtain an occurrence of $\mathbf{u}$ in $\mathbf{w}$ with gaps $G''_1,G''_2,\dots,G''_{k+1}$ such that $G''_\ell\subseteq G'_\ell$ for each $\ell=1,2,\dots,k+1$. The same argument shows that it must be the leftmost occurrence of $\mathbf{u}$ in $\mathbf{w}$. As discussed prior to the formulation of this corollary, the leftmost occurrence is unique so that the occurrence of $\mathbf{u}$ in $\mathbf{w}$ with gaps $G''_1,G''_2,\dots,G''_{k+1}$ coincides with the occurrence \eqref{eq:leftmost1} with gaps $G_1,G_2,\dots,G_{k+1}$. Hence, $G_\ell=G''_\ell\subseteq G'_\ell\subseteq G_\ell$ for each $\ell=1,2,\dots,k+1$, and thus, the leftmost occurrences of $\mathbf{u}$ in $\mathbf{w}$ and $\mathbf{w'}$ have the same gaps.
\end{proof}

\begin{remark} As a brief digression from our discussion, we record the following consequence of Corollary~\ref{cor:leftmost}.
\begin{corollary}
\label{cor:subword-n}
If an identity $\mathbf{w}\approx\mathbf{w'}$ holds in the semigroup $(T_n,\cdot)$, then the words $\mathbf{w}$ and $\mathbf{w'}$ have the same subwords of length $n$.
\end{corollary}

\begin{proof}
Let a word $\mathbf{u}=u_1\cdots u_n$ occur in $\mathbf{w}$, and let
\[
\mathbf{w} = \mathbf{w}_1u_1\mathbf{w}_2\cdots \mathbf{w}_{n-1}u_{n-1}\mathbf{w}_{n}
\]
be the leftmost occurrence of the word $u_1\cdots u_{n-1}$ in $\mathbf{w}$. Then the variable $u_n$ occurs in the suffix $w_n$. By Corollary~\ref{cor:leftmost}, the leftmost occurrence of $u_1\cdots u_{n-1}$ in $\mathbf{w'}$,
\[
\mathbf{w'} = \mathbf{w'}_1u_1\mathbf{w'}_2\cdots \mathbf{w'}_{n-1}u_{n-1}\mathbf{w'}_{n}
\]
has the same gaps; in particular,  $\alf(\mathbf{w'}_n)=\alf(\mathbf{w}_n)$. Hence, the variable $u_n$ occurs in the suffix $\mathbf{w'}_n$, and the word $\mathbf{u}$ occurs in $\mathbf{w'}$.
\end{proof}

The result of Corollary~\ref{cor:subword-n} was somewhat unexpected, as it relates identities of Boolean upper triangular $n\times n$-matrices to those of Boolean upper \textbf{unitriangular} matrices of larger size. (A triangular matrix is called \emph{unitriangular} if all of its diagonal entries equal 1.) Let $U_{n+1}$ stand for the set of all Boolean upper unitriangular $(n+1)\times(n+1)$-matrices; clearly, this set is closed under matrix addition and multiplication. The identities of the semigroup $(U_{n+1},\cdot)$ were characterized in \cite[Theorem 2]{Vo04}: an identity $\mathbf{w}\approx\mathbf{w'}$ holds in $(U_{n+1},\cdot)$ if and only if the words $\mathbf{w}$ and $\mathbf{w'}$ have the same subwords of length at most $n$. Consequently, Corollaries~\ref{cor:subword-n} and~\ref{cor:ais-identity} imply that the semigroup $(U_{n+1},\cdot)$ satisfies every identity holding in the semigroup $(T_n,\cdot)$. The converse, however, fails for every $n>1$; for instance, the identity $(xy)^n\approx(yx)^n$ holds in $(U_{n+1},\cdot)$, but is violated in $(T_n,\cdot)$. We also notice that the established relation does not extend to semiring identities: for instance, the identity $xy\approx x^2y+xy^2$ holds in the \ais{} $(T_2,+,\cdot)$; see Example~\ref{ex:two identities}(ii), but this identity fails in the \ais{} $(U_3,+,\cdot)$.
\end{remark}

Returning to our discussion of the condition $\exists\mathrm{EG}_n$, we observe that, contrary to what one might expect, this condition is not sufficient for an identity to hold in the semigroup $(T_n,\cdot)$. For instance, the identity $xyx^2yx\approx xyxyx$ fails in $(T_3,\cdot)$ since, under the substitution
\[
x\mapsto\begin{pmatrix}1&1&0\\0&0&1\\0&0&1\end{pmatrix},\quad y\mapsto\begin{pmatrix}1&0&0\\0&0&0\\0&0&1\end{pmatrix},
\]
the value of the word $xyx^2yx$ is the matrix $\left(\begin{smallmatrix}1&1&1\\0&0&1\\0&0&1\end{smallmatrix}\right)$ while the value of the word $xyxyx$ is the matrix $\left(\begin{smallmatrix}1&1&0\\0&0&1\\0&0&1\end{smallmatrix}\right)$. (Alternatively, one can use Corollary~\ref{cor:ais-identity}: the word $xyx^2yx$ has an occurrence of the subword $x^2$ with gaps $\{x,y\},\varnothing,\{x,y\}$ while the subword does not occur in $xyxyx$ with empty `middle' gap.)  At the same time, a direct inspection shows that the identity satisfies the condition $\rule{0pt}{12pt}\exists\mathrm{EG}_3$, and moreover, for every common subword of  $xyx^2yx$ and $xyxyx$ with length at most 2, its leftmost occurrences in these words have the same gaps (and so do its dually defined rightmost occurrences).

On the other hand, by Corollary~\ref{cor:ais-identity}, the following condition $\forall\mathrm{EG}_n$ (all gaps are equal) is sufficient for an identity $\mathbf{w}\approx\mathbf{w'}$ to hold in $(T_n,\cdot)$:
\begin{itemize}
  \item[$\forall\mathrm{EG}_n$:] for every occurrence of a subword of length $k<n$ in one of the words $\mathbf{w}$ and $\mathbf{w'}$, there exists an occurrence of the subword in the other word with the same gaps.
\end{itemize}
However, the condition $\forall\mathrm{EG}_n$ is not necessary. For instance, the identity $x^2yx\approx xyx$ is easily seen to hold in the semigroup $(T_2,\cdot)$ and at the same time, the word $x$ of length 1 occurs with the gaps $\{x\},\{x,y\}$ in the word $x^2yx=x\cdot\underline{x}\cdot yx$, but not in the word $xyx$. So, $\forall\mathrm{EG}_2$ fails for this identity.

\subsection{A simplification in the semigroup case}
We aim to slightly simplify the criterion of Theorem~\ref{thm:ais-inequality} for semigroup inequalities. For this, we need two lemmas.

\begin{lemma}
\label{lem:long words}
Let $k$ be a positive integer and $\mathbf{w}$ a word of length at least $k$. Suppose that a word $\mathbf{w'}$ is such that every word of length $k$ that occurs in $\mathbf{w}$ with gaps $G_1,G_2,\dots,G_{k+1}$ also occurs in $\mathbf{w'}$ with gaps $G'_1,G'_2,\dots,G'_{k+1}$ such that $G'_\ell\subseteq G_\ell$ for each $\ell=1,2,\dots,k+1$. Then every word of length $k-1$ that occurs in $\mathbf{w}$ with gaps $H_1,H_2,\dots,H_k$ occurs in $\mathbf{w'}$ with gaps $H'_1,H'_2,\dots,H'_k$ such that $H'_j\subseteq H_j$ for each $j=1,2,\dots,k$.
\end{lemma}

\begin{proof}
Let $\mathbf{u}=u_1\cdots u_{k-1}$ occur as a  subword in $\mathbf{w}$ with gaps $H_1,H_2,\dots,H_k$. This means that there exist words $\mathbf{w}_1, \mathbf{w}_2, \dots,  \mathbf{w}_k$ (some of which may be empty) such that
\[
\mathbf{w} = \mathbf{w}_1 u_1 \mathbf{w}_2\cdots u_{k-1}\mathbf{w}_k
\]
and $H_i=\alf(\mathbf{w}_i)$ for each $i=1,2,\dots,k$. As $|\mathbf{w}|\ge k$, at least one of the gaps $H_1,H_2,\dots,H_k$ is nonempty. Fix an index $i$ such that the gap $H_i$ is nonempty and pick up an arbitrary variable $t\in H_i$. Then the word $\mathbf{w}_i$ can be decomposed as $\mathbf{w}_i=\mathbf{y}t\mathbf{z}$ for some (possibly empty) words $\mathbf{y}$ and $\mathbf{z}$. The word
\[
\mathbf{v}:=\begin{cases} tu_1\cdots u_{k-1}&\text{if \ $i=1$,}\\
u_1\cdots u_{i-1}tu_i\cdots u_{k-1}&\text{if \ $1<i<k$,}\\
u_1\cdots u_{k-1}t&\text{if \ $i=k$}
\end{cases}
\]
has length $k$ and occurs as a subword in the word $\mathbf{w}$ with gaps $G_1,G_2,\dots,G_{k+1}$, where for each $\ell=1,2,\dots,k+1$,
\[
G_\ell=\begin{cases}
H_\ell&\text{if \ $\ell<i$,}\\
\alf(\mathbf{y})&\text{if \ $\ell=i$,}\\
\alf(\mathbf{z})&\text{if \ $\ell=i+1$,}\\
H_{\ell-1}&\text{if \ $\ell>i+1$.}\\
\end{cases}
\]
The top row of the scheme below shows the occurrences of $\mathbf{u}$ and $\mathbf{v}$ in the word $\mathbf{w}$ and the gaps corresponding to these occurrences for the case $1<i<k$. In this scheme, a symbol placed over or under a brace denotes the alphabet of the overbraced or underbraced word.
\begin{center}
\begin{tabular}{rccccccccccccc}
$\mathbf{w}=$ & $\underbrace{\overbrace{\mathbf{w}_1}^{H_1}}_{G_1}$ & $\cdot$ & $u_1$ & \dots & $u_{i-1}$ & $\cdot$ & $\overbrace{\underbrace{\mathbf{y}}_{G_i}\cdot\ t\cdot \underbrace{\mathbf{z}}_{G_{i+1}}}^{H_i}$ & $\cdot$ & $u_i$ & \dots & $u_k$ & $\cdot$ & $\underbrace{\overbrace{\mathbf{w}_k}^{H_k}}_{G_{k+1}}$ \\
&\ \begin{rotate}{90}$\subseteq$\end{rotate} &&& \dots &&&\ \begin{rotate}{90}$\subseteq$\end{rotate} \qquad\quad  \begin{rotate}{90}$\subseteq$\end{rotate}&&& \dots &&&\ \begin{rotate}{90}$\subseteq$\end{rotate} \\
$\mathbf{w'}=$ & $\underbrace{\overbrace{\mathbf{w'}_1}^{G'_1}}_{H'_1}$ & $\cdot$ & $u_1$ & \dots & $u_{i-1}$ & $\cdot$ & $\underbrace{\overbrace{\mathbf{y'}}^{G'_i}\cdot\ t\cdot \overbrace{\mathbf{z'}}^{G'_{i+1}}}_{H'_i}$ & $\cdot$ & $u_i$ & \dots & $u_k$ & $\cdot$ & $\underbrace{\overbrace{\mathbf{w'}_k}^{G'_{k+1}}}_{H'_k}$
\end{tabular}
\end{center}
By the condition of the lemma, the word $\mathbf{v}$ occurs as a subword in the word $\mathbf{w'}$ with gaps $G'_1,G'_2,\dots,G'_{k+1}$ such that $G'_\ell\subseteq G_\ell$ for each $\ell=1,2,\dots,k+1$.
The occurrence of $\mathbf{v}$ in $\mathbf{w'}$ induces an occurrence of $\mathbf{u}$ in $\mathbf{w'}$ with gaps $H'_1,H'_2,\dots,H'_k$, where for each $j=1,2,\dots,k$,
\[
H'_j=\begin{cases}
G'_j&\text{if \ $j<i$,}\\
G'_i\cup\{t\}\cup G'_{i+1}&\text{if \ $j=i$,}\\
G'_{j+1}&\text{if \ $j>i$.}
\end{cases}
\]
For the case $1<i<k$, the occurrences of $\mathbf{u}$ and $\mathbf{v}$ in the word $\mathbf{w'}$ and the gaps corresponding to these occurrences are shown in the bottom row of the scheme above. Clearly, the inclusions  $G'_\ell\subseteq G_\ell$ for $\ell=1,2,\dots,k+1$ imply the inclusions $H'_j\subseteq H_j$ for $j=1,2,\dots,k$.
\end{proof}

A word $\mathbf{w}$ is \emph{minimal} [\emph{maximal}] for an ordered semigroup $(S,\cdot,\le)$ if the only word $\mathbf{w'}$ such that $(S,\cdot,\le)$ satisfies the inequality $\mathbf{w'}\preccurlyeq\mathbf{w}$ [respectively, $\mathbf{w}\preccurlyeq\mathbf{w'}$] is the word $\mathbf{w}$ itself.

\begin{lemma}
\label{lem:short words}
Words of length less than $n$ are maximal for the ordered semigroup $(T_n,\cdot,\le)$.
\end{lemma}

\begin{proof}
Take any word $\mathbf{w}=w_1\cdots w_k$ with $k<n$, and suppose that a word $\mathbf{w'}$ is such that the inequality $\mathbf{w}\preccurlyeq\mathbf{w'}$ holds in $(T_n,\cdot,\le)$. The word $\mathbf{w}$ occurs as a  subword of itself with all gaps $G_1,G_2,\dots,G_{k+1}$ being empty. By Theorem~\ref{thm:ais-inequality}, $\mathbf{w}$ must occur as a subword in $\mathbf{w'}$ with gaps $G'_1,G'_2,\dots,G'_{k+1}$ such that $G'_\ell\subseteq G_\ell$ for each $\ell=1,2,\dots,k+1$. Since $G_\ell$ is empty, so is $G'_\ell$. Thus, there is an occurrence of $w_1\cdots w_k$ as a subword in $\mathbf{w'}$ in which nothing precedes $w_1$, nothing follows $w_k$, and nothing occurs between $w_i$ and $w_{i+1}$ for each $i=1,2,\dots,k-1$. This is only possible if $\mathbf{w'}=w_1\cdots w_k$.
\end{proof}

\begin{theorem}
\label{thm:sem-inequality}
The ordered semigroup of all Boolean upper triangular $n\times n$-matrices satisfies an inequality $\mathbf{w}\preccurlyeq\mathbf{w'}$ if and only if either $\mathbf{w}=\mathbf{w'}$ or $|\mathbf{w}|\ge n$ and every word of length $n-1$ that occurs in $\mathbf{w}$ with gaps $G_1,G_2,\dots,G_n$ occurs in $\mathbf{w'}$ with gaps $G'_1,G'_2,\dots,G'_n$ such that $G'_\ell\subseteq G_\ell$ for each $\ell=1,2,\dots,n$.
\end{theorem}

\begin{proof}
\emph{Necessity.} Suppose that $\mathbf{w}\ne\mathbf{w'}$. Then Lemma~\ref{lem:short words} ensures that $|\mathbf{w}|\ge n$. That every word of length $n-1$ occurring in $\mathbf{w}$ with gaps $G_1,G_2,\dots,G_n$ occurs in $\mathbf{w'}$ with gaps $G'_1,G'_2,\dots,G'_n$ such that $G'_\ell\subseteq G_\ell$ for each $\ell=1,2,\dots,n$ is a special instance of the condition established in Theorem~\ref{thm:ais-inequality}.

\medskip

\noindent\emph{Sufficiency.} If $\mathbf{w}=\mathbf{w'}$, there is nothing to prove. So, suppose that $|\mathbf{w}|\ge n$ and every word of length $n-1$ that occurs in $\mathbf{w}$ with gaps $G_1,G_2,\dots,G_n$ occurs in $\mathbf{w'}$ with gaps $G'_1,G'_2,\dots,G'_n$ such that $G'_\ell\subseteq G_\ell$ for each $\ell=1,2,\dots,n$. Applying Lemma~\ref{lem:long words} for $k=n-1,n-2,\dots,1$ in succession, we get that every word of length $k<n$ occurring in $\mathbf{w}$ with gaps $H_1,H_2,\dots,H_{k+1}$ also occurs in $\mathbf{w'}$ with gaps $H'_1,H'_2,\dots,H'_{k+1}$ such that $H'_j\subseteq H_j$ for each $j=1,2,\dots,k{+}1$. Then the inequality $\mathbf{w}\preccurlyeq\mathbf{w'}$ holds in $(T_n,\cdot,\le)$ by Theorem~\ref{thm:ais-inequality}.
\end{proof}

A characterization of identities holding in the semigroup $(T_n,\cdot)$ is immediate.

\begin{corollary}
\label{cor:sem-identity}
The semigroup of all Boolean upper triangular $n\times n$-matrices satisfies an identity $\mathbf{w}\approx\mathbf{w'}$ if and only if either $\mathbf{w}=\mathbf{w'}$ or $|\mathbf{w}|,|\mathbf{w'}|\ge n$ and every word of length $n-1$ that occurs in one of the words $\mathbf{w}$ and $\mathbf{w'}$ with gaps $G_1,G_2,\dots,G_n$ occurs in the other word with gaps $G'_1,G'_2,\dots,G'_n$ such that $G'_\ell\subseteq G_\ell$ \ for each $\ell=1,2,\dots,n$.
\end{corollary}

\section{Applications}
\label{sec:applications}

\subsection{Complexity of identity checking}\label{subsec:check-id}\label{subsec:complexity} Here we assume the reader's familiarity with a few basic notions of computational complexity; they all can be found, e.g., in the early chapters of Papadimitriou's textbook \cite{Pa94}.

Given a semigroup $(S,\cdot)$, its \emph{identity checking problem}\footnote{Also called the `\emph{term equivalence problem}' in the literature.} \textsc{Check-Id}$(S,\cdot)$ is the following decision problem. An instance of \textsc{Check-Id}$(S,\cdot)$ is a semigroup identity $\mathbf{w}\approx\mathbf{w'}$. The answer to the instance is `YES' whenever the identity $\mathbf{w}\approx\mathbf{w'}$ holds in $(S,\cdot)$; otherwise, the answer is `NO'.
We stress that here the semigroup $(S,\cdot)$ is fixed and it is the identity $\mathbf{w}\approx\mathbf{w'}$ that serves as the input so that the time/space complexity of \textsc{Check-Id}$(S,\cdot)$ is measured in terms of the size of the identity, that is, in $|\mathbf{w}\mathbf{w'}|$. Studying computational complexity of identity checking in semigroups was proposed by Mark Sapir \cite[Problem~2.4]{KS95}.

For a finite semigroup $(S,\cdot)$, the problem \textsc{Check-Id}$(S,\cdot)$ belongs to the complexity class $\mathsf{coNP}$. Indeed, if an identity $\mathbf{w}\approx\mathbf{w'}$ fails in $(S,\cdot)$ and the words $\mathbf{w}$, $\mathbf{w'}$ involve $m$ variables in total, then a nondeterministic algorithm guesses an $m$-tuple of elements in $S$ witnessing the failure and then verifies the guess by computing the values of the words $\mathbf{w}$ and $\mathbf{w'}$ under the substitution that sends the variables in $\alf(\mathbf{w})\cup\alf(\mathbf{w}')$ to the entries of the guessed $m$-tuple. With multiplication in $(S,\cdot)$ assumed to be performed in unit time, the algorithm takes linear in $|\mathbf{w}\mathbf{w'}|$ time.

There are numerous examples of finite semigroups (and even finite groups) whose identity checking problem is $\mathsf{coNP}$-complete. We refer the reader to \cite{AVG09} and references therein for specific examples and a general discussion of the complexity viewpoint on verifying identities in finite semigroups. Here we only mention that the identity checking problem for the semigroup of \textbf{all} Boolean $n\times n$-matrices is $\mathsf{coNP}$-complete whenever $n\ge 5$. This does not seem to have been explicitly stated in the literature, but is an immediate consequence of the two following facts:
\begin{itemize}
  \item if a finite semigroup possesses a nonsolvable subgroup, then identity checking in the semigroup is $\mathsf{coNP}$-complete \cite[Corollary 1]{AVG09};
  \item for each $n\ge 5$, the semigroup of all Boolean $n\times n$-matrices has a nonsolvable subgroup, namely, the subgroup consisting of all permutation $n\times n$-matrices.
\end{itemize}

In contrast, Corollary~\ref{cor:sem-identity} leads to the following result:
\begin{proposition}
\label{prop:polynomial}
For any $n$, there is a polynomial time algorithm to decide if an identity holds in the semigroup of all upper triangular Boolean $n\times n$-matrices.
\end{proposition}

\begin{proof}
Take an arbitrary identity $\mathbf{w}\approx\mathbf{w'}$ and let $m:=|\mathbf{w}\mathbf{w'}|$. To fix an occurrence of a subword of length $n-1$ in a word, one has to choose $n-1$ positions in the word. Hence, the number of all possible occurrences of subwords of length $n-1$ in each of the words $\mathbf{w}$ and $\mathbf{w'}$ is at most $\binom{m-1}{n-1}$. For each occurrence, one scan of $\mathbf{w}$ and $\mathbf{w'}$ suffices to compute corresponding gaps. Thus, in time $\le2(m-1)\binom{m-1}{n-1}$, which is a polynomial of $m$, we can compose two lists containing all subwords of length $n-1$ occurring in $\mathbf{w}$ and $\mathbf{w'}$ with gaps for each occurrence. Comparing the lists for $\mathbf{w}$ and $\mathbf{w'}$, we can verify the condition of Corollary~\ref{cor:sem-identity}, and hence, we can check if the identity $\mathbf{w}\approx\mathbf{w'}$ holds in the semigroup $(T_n,\cdot)$  in time polynomial in $m$.
\end{proof}

To place Proposition~\ref{prop:polynomial} in a broader context, we mention an analogous result for semigroups of upper triangular tropical $n\times n$-matrices. The \emph{tropical} ai-semiring is formed by the real numbers augmented with the symbol $-\infty$ under the operations $a\oplus b:=\max\{a,b\}$ and $a\otimes b:=a+b$, for which $-\infty$ plays the role of zero: $a\oplus-\infty=-\infty\oplus a=a$ and $a\otimes-\infty=-\infty\otimes a=-\infty$. A square matrix over the tropical \ais{} is considered \emph{upper triangular} if its entries below the main diagonal are all $-\infty$. The identity checking problem for the semigroup of all upper triangular tropical $n\times n$-matrices was deeply studied in \cite{DJK18,JT19} and was proved polynomial time decidable. Compared to our algorithm for \textsc{Check-Id}$(T_n,\cdot)$, algorithms developed in \cite{DJK18,JT19} are more cumbersome and their polynomial-time implementation crucially depends on the existence of a polynomial time algorithm for (real) linear programming. Still, the underlying logic is the same and the use of triangularity in our algorithm and those in \cite{DJK18,JT19} is quite similar.

What we have just said should not be understood as claiming that semigroups consisting of upper triangular matrices always admit polynomial time algorithms for checking identities. To show that not all semigroups of triangular matrices have identity checking problem in $\mathsf{P}$, we exhibit a subsemigroup $(C_4,\cdot)$ of $(T_4,\cdot)$ with $\mathsf{coNP}$-complete identity checking. The set $C_4$ comprises the following nine matrices:
\begin{align*}
e&=\begin{pmatrix}
1 & 0 & 0 & 0\\
0 & 1 & 0 & 0\\
0 & 0 & 1 & 0\\
0 & 0 & 0 & 1
\end{pmatrix},\quad
&a&=\begin{pmatrix}
1 & 1 & 0 & 0\\
0 & 0 & 0 & 1\\
0 & 0 & 0 & 0\\
0 & 0 & 0 & 1
\end{pmatrix},\quad
&b&=\begin{pmatrix}
1 & 0 & 1 & 0\\
0 & 0 & 0 & 0\\
0 & 0 & 0 & 1\\
0 & 0 & 0 & 1
\end{pmatrix},\\
a^2&=\begin{pmatrix}
1 & 1 & 0 & 1\\
0 & 0 & 0 & 1\\
0 & 0 & 0 & 0\\
0 & 0 & 0 & 1
\end{pmatrix},\quad
&ab&=\begin{pmatrix}
1& 0& 1& 0\\
0& 0& 0& 1\\
0& 0& 0& 0\\
0& 0& 0& 1
\end{pmatrix},\quad
&ba&=\begin{pmatrix}
1 & 1 & 0 & 0\\
0 & 0 & 0 & 0\\
0 & 0 & 0 & 1\\
0 & 0 & 0 & 1
\end{pmatrix},\\
b^2&=\begin{pmatrix}
1 & 0 & 1 & 1\\
0 & 0 & 0 & 0\\
0 & 0 & 0 & 1\\
0 & 0 & 0 & 1
\end{pmatrix},\quad
&a^2b&=\begin{pmatrix}
1 & 0 & 1 & 1\\
0 & 0 & 0 & 1\\
0 & 0 & 0 & 0\\
0 & 0 & 0 & 1
\end{pmatrix},\quad
&ba^2&=\begin{pmatrix}
1 & 1 & 0 & 1\\
0 & 0 & 0 & 0\\
0 & 0 & 0 & 1\\
0 & 0 & 0 & 1
\end{pmatrix}.
\end{align*}

We register several properties of $(C_4,\cdot)$. They all can be verified by direct computations.
\begin{enumerate}
  \item[$0^\circ$] The equalities $a^3=a^2$, $b^3=b^2$, and $ab^2=a^2b=a^2b^2$ hold.
  \item[$1^\circ$] The equalities $aba=a$ and $bab=b$ hold.
  \item[$2^\circ$] Except for $a$ and $b$, all matrices $c\in C_4$ are \emph{idempotent}, that is, satisfy $c^2=c$.
  \item[$3^\circ$] For each $c\in C_4$, its square $c^2$ is idempotent.
  \item[$4^\circ$] Any product of idempotent matrices from $C_4$ is idempotent.
  \item[$5^\circ$] The set $I:=\{a^2,b^2,a^2b,ba^2\}$ of all matrices having 1 in the position $(1,4)$ is an \emph{ideal} in $(C_4,\cdot)$, that is, $cd,dc\in I$ for all $c\in C_4$ and $d\in I$.
\end{enumerate}

\begin{proposition}
\label{prop:c4}
The problem \textsc{Check-Id}$(C_4,\cdot)$ is $\mathsf{coNP}$-complete.
\end{proposition}

\begin{proof}
As mentioned at the beginning of Subsect.~\ref{subsec:complexity}, the problem \textsc{Check-Id}$(S,\cdot)$ lies in the class $\mathsf{coNP}$ for any finite semigroup $(S,\cdot)$. Hence, it suffices to show that the problem \textsc{Check-Id}$(C_4,\cdot)$ is $\mathsf{coNP}$-hard. For this, we build a polynomial time reduction to the negation of \textsc{Check-Id}$(C_4,\cdot)$ from the problem \textsc{Hitting Set} which is known to be $\mathsf{NP}$-complete \cite[Main Theorem, item 15]{Ka72}.

An instance of \textsc{Hitting Set} is a family $\{U_i\mid i=1,\dots,q\}$ of subsets of a finite set $S=\{s_j\mid j=1,\dots,r\}$, and the question is whether there exists a \emph{hitting set} $H\subseteq S$ such that $|H\cap U_i|=1$ for each $i=1,\dots,q$. We are going to construct words $\mathbf{w}$ and $\mathbf{w'}$ such that:

1)~$|\mathbf{w}\mathbf{w'}|$ is upper bounded by a polynomial in $r$ and $q$, and

2) the identity $\mathbf{w}\approx\mathbf{w'}$ fails in the semigroup $(C_4,\cdot)$ if and only if a hitting set for the family $\{U_i\}$ exists.

We introduce $2r+1$ variables: two variables $x_j$ and $y_j$ for each $j=1,\dots,r$ and an extra variable $z$. For each $j=1,\dots,r$, let $\mathbf{v}_j:=(x_jy_jzy_jx_jz)^2$. Further, for each $i=1,\dots,q$, if $U_i=\{s_{j_1},\dots,s_{j_p}\}$ where $p=|U_i|\le r$, then we let $\mathbf{u}_i:=(x_{j_1}\cdots x_{j_p}z)^2$. Finally, define
\[
\mathbf{w}:=z\mathbf{u}_1\cdots \mathbf{u}_q\mathbf{v}_1\cdots\mathbf{v}_r\ \text{ and }\ \mathbf{w'}:=\mathbf{w}^2.
\]
Since $|\mathbf{v}_j|=12$ for each $j=1,\dots,r$ and $|\mathbf{u}_i|\le 2(r+1)$ for each $i=1,\dots,q$, we conclude that $|\mathbf{w}\mathbf{w'}|\le 3(1+2(r+1)q+12r)$, so claim 1) is established.

For the `if' part of claim 2), let $H$ be a hitting set for the family $\{U_i\mid i=1,\dots,q\}$. Define a substitution $\varphi\colon\{x_1,y_1,\dots,x_r,y_r,z\}\to C_4$ as follows:
\[
x_j\varphi:=\begin{cases}
b&\text{if $s_j\in H$},\\
e&\text{if $s_j\notin H$};
\end{cases}\qquad
y_j\varphi:=\begin{cases}
e&\text{if $s_j\in H$},\\
b&\text{if $s_j\notin H$};
\end{cases}\qquad
z\varphi:=a.
\]
One readily computes that $\mathbf{v}_j\varphi=(ba)^4$ for all $j=1,\dots,r$. Moreover, since $|H\cap U_i|=1$ for each $i=1,\dots,q$, one has $\mathbf{u}_i\varphi=(ba)^2$. By property $1^\circ$ one gets $\mathbf{w}\varphi=a(ba)^{2q}(ba)^{4r}=a$. Hence, the matrix $\mathbf{w}\varphi$ is not idempotent and the identity $\mathbf{w}\approx\mathbf{w'}$ fails in $(C_4,\cdot)$.

For the `only if' part, suppose that $\mathbf{w}\approx\mathbf{w'}$ fails in the semigroup $(C_4,\cdot)$. Then there exists a substitution $\varphi\colon\{x_1,y_1,\dots,x_r,y_r,z\}\to C_4$ with $\mathbf{w}\varphi\ne\mathbf{w'}\varphi$ which means that the matrix $\mathbf{w}\varphi$ is not idempotent. By the construction, all matrices $\mathbf{u}_i\varphi$ and $\mathbf{v}_j\varphi$ are squares whence they are idempotent
(property $3^\circ$). Since any product of idempotent matrices from $C_4$ is idempotent (property $4^\circ$), for $\mathbf{w}\varphi$ not to be idempotent, the matrix $z\varphi$ must not be idempotent. By property $2^\circ$ this means either $z\varphi=a$ or $z\varphi=b$. Consider the case $z\varphi=a$; the other case is fully analogous.

All matrices in the set $I=\{a^2,b^2,a^2b,ba^2\}$ are idempotent while $\mathbf{w}\varphi$ is not. Hence $\mathbf{w}\varphi\notin I$, and since $I$ is an ideal (property $5^\circ$), no factor of $\mathbf{w}\varphi$ can lie in $I$. In particular,
\begin{equation}\label{eq:elements}
\mathbf{v}_j\varphi=\left((x_jy_jzy_jx_jz)^2\right)\varphi=\left((x_j\varphi)(y_j\varphi)a(y_j\varphi)(x_j\varphi)a\right)^2\notin I,
\end{equation}
for every  $j=1,\dots,r$. For \eqref{eq:elements} to hold, one of the matrices $x_j\varphi,y_j\varphi$ must be $b$ and the other must be $e$. Indeed, both $x_j\varphi$ and $y_j\varphi$ must lie in $C_4\setminus I=\{e,a,b,ab,ba\}$. If $x_j\varphi$ or $y_j\varphi$ lies in $\{a,ab,ba\}$, then $a^2$ would occur as a factor in $\left((x_j\varphi)(y_j\varphi)a(y_j\varphi)(x_j\varphi)a\right)^2$, contradicting \eqref{eq:elements}. Hence $x_j\varphi,y_j\varphi\in\{b,e\}$ but the matrices $x_j\varphi,y_j\varphi$ cannot be equal as either $a^2$ or $b^2$ would occur as a factor in $\left((x_j\varphi)(y_j\varphi)a(y_j\varphi)(x_j\varphi)a\right)^2$ otherwise.

Thus, $x_j\varphi\in\{b,e\}$ for every $j$. For each $i=1,\dots,q$, we have
\begin{equation}\label{eq:sets}
\mathbf{u}_i\varphi=\left((x_{j_1}\cdots x_{j_p}z)^2\right)\varphi=\left((x_{j_1}\varphi)\cdots (x_{j_p}\varphi)a\right)^2\notin I.
\end{equation}
If at least two of the matrices $x_{j_k}\varphi$, $k=1,\dots,p$, are equal to $b$, then $b^2$ occurs as a factor of $\left((x_{j_1}\varphi)\cdots (x_{j_p}\varphi)a\right)^2$, and if $x_{j_k}\varphi=e$ for all $k=1,\dots,p$, then $\left((x_{j_1}\varphi)\cdots (x_{j_p}\varphi)a\right)^2=a^2$. Both options contradict \eqref{eq:sets}. Hence, \eqref{eq:sets} implies that exactly one  of the matrices $x_{j_k}\varphi$, $k=1,\dots,p$, is equal to $b$. Therefore, the set $H:=\{s_j\in S\mid x_j\varphi=b\}$ has a singleton intersection with each set $U_i$, $i=1,\dots,q$, so $H$ is a hitting set for the family $\{U_i\}$.
\end{proof}

\begin{remark}
The semigroup $(C_4,\cdot)$ is tightly related to the 6-element \emph{Brandt semigroup} $(B_2^1,\cdot)$, which is, quoting from \cite{JZ21}, `perhaps the most ubiquitous harbinger of complex behaviour in all finite semigroups'. Here $B_2^1$ stands for the set comprised of the following six $2\times 2$-matrices:
\[
\begin{pmatrix}
1 & 0 \\ 0 & 1
\end{pmatrix}, ~~
\begin{pmatrix}
1 & 0 \\ 0 & 0
\end{pmatrix}, ~~
\begin{pmatrix}
0 & 1 \\ 0 & 0
\end{pmatrix}, ~~
\begin{pmatrix}
0 & 0 \\ 1 & 0
\end{pmatrix}, ~~
\begin{pmatrix}
0 & 0 \\ 0 & 1
\end{pmatrix}, ~~
\begin{pmatrix}
0 & 0 \\ 0 & 0
\end{pmatrix}
\]
and $\cdot$ is the usual matrix multiplication. The map $C_4\to B_2^1$ defined by
\begin{gather*}
e\mapsto\begin{pmatrix}
1 & 0 \\ 0 & 1
\end{pmatrix}, ~~
ab\mapsto\begin{pmatrix}
1 & 0 \\ 0 & 0
\end{pmatrix}, ~~
a\mapsto\begin{pmatrix}
0 & 1 \\ 0 & 0
\end{pmatrix}, ~~
b\mapsto\begin{pmatrix}
0 & 0 \\ 1 & 0
\end{pmatrix}, ~~
ba\mapsto\begin{pmatrix}
0 & 0 \\ 0 & 1
\end{pmatrix},\\
a^2,b^2,a^2b,ba^2\mapsto\begin{pmatrix}
0 & 0 \\ 0 & 0
\end{pmatrix}
\end{gather*}
is easily seen to be a semigroup homomorphism, and this homomorphism induces an isomorphism between the Rees quotient of $(C_4,\cdot)$ over the ideal $\{a^2,b^2,a^2b,ba^2\}$ and $(B_2^1,\cdot)$.

It was proved by Seif~\cite{Se05} and, independently, Kl\'{\i}ma \cite{Kl09} that the semigroup $(B_2^1,\cdot)$ has $\mathsf{coNP}$-complete identity checking. Kl\'{\i}ma's argument carries over $(C_4,\cdot)$ with minimal adjustments. Still, for the reader's convenience, we have supplied Proposition~\ref{prop:c4} with a complete proof. We reduce from an $\mathsf{NP}$-complete problem other than that utilized by Kl\'{\i}ma, but re-use Kl\'{\i}ma's construction of semigroup words employed in the reduction.
\end{remark}

Studying the complexity of the identity checking problem is also interesting for algebras other than semigroups. In the context of the present paper, it appears natural to look at the class of finite \ais{}s.

Dealing with identity checking in the presence of the distributive law (as in rings or semirings) requires particular care in defining what is meant by the size of an instance. For example, the expression
\begin{equation}\label{eq:n-foldproduct}
(x_1 + y_1)\cdots(x_n + y_n)
\end{equation}
involves just $2n$ variables, yet when expanded using distributivity, it becomes a sum of $2^{n}$ words, each of length $n$. Should the size of \eqref{eq:n-foldproduct} be considered $2n$ or $n2^n$? Here we adopt the convention consistent with our definition of a semiring polynomial as a sum of words: the size of an instance $\mathbf{W}\approx\mathbf{W'}$ of the problem \textsc{Check-Id}$(T_n,+,\cdot)$, where $\mathbf{W}=\mathbf{w}_1+\dots+\mathbf{w}_p$ and $\mathbf{W'}=\mathbf{w}'_1+\dots+\mathbf{w}'_{p'}$ are semiring polynomials, is defined as
\[
\sum_{i=1}^p|\mathbf{w}_i|+\sum_{i=1}^{p'}|\mathbf{w}'_i|,
\]
that is, the total length of all words $\mathbf{w}_1,\dots,\mathbf{w}_p,\mathbf{w}'_1,\dots,\mathbf{w}'_{p'}$. Under this convention, checking semiring identities of triangular Boolean $n\times n$-matrices is almost as easy as checking semigroup identities.

\begin{proposition}
\label{prop:polynomial-ais}
For any $n$, there is a polynomial time algorithm to decide if an identity holds in the semiring of all upper triangular Boolean $n\times n$-matrices.
\end{proposition}

\begin{proof} We adapt the argumentation from the proof of Proposition~\ref{prop:polynomial}. Take an arbitrary identity $\mathbf{W}\approx\mathbf{W'}$ and let $m$ be its size as defined above. Since $\mathbf{W}$ and $\mathbf{W'}$ each contain at least one word of length $\ge1$, the total length of the words in each polynomial is at most $m-1$. Hence both the number of words in each of the semiring polynomials $\mathbf{W}$ and $\mathbf{W'}$ and the length of every such word are at most $m-1$. To specify an occurrence of a subword of length $k<n$ in a word, one has to choose $k$ positions in that word. Hence, the number of all possible occurrences of subwords of length $k< n-1$ in every word in each of the semiring polynomials $\mathbf{W}$ and $\mathbf{W'}$ is at most $\sum_{k=0}^{n-1}\binom{m-1}{k}$. For each such occurrence, a single scan of the word suffices to compute corresponding gaps.  Thus, in time $\le2(m-1)^2\sum_{k=0}^{n-1}\binom{m-1}{k}$, which is a polynomial of $m$, we can compose two lists containing all subwords of length $k<n$ occurring in $\mathbf{W}$ and $\mathbf{W'}$ with gaps for each occurrence. Comparing the lists for $\mathbf{W}$ and $\mathbf{W'}$ lets us verify the condition of Corollary~\ref{cor:ais-identity}, and hence, we can check  whether the identity $\mathbf{W}\approx\mathbf{W'}$ holds in the \ais{} $(T_n,+,\cdot)$ in time polynomial in $m$.
\end{proof}

In the above proof of Proposition~\ref{prop:polynomial-ais} we had to consider subwords of all lengths $<n$ and not only those of length $n-1$ as we did when proving Proposition~\ref{prop:polynomial}. This is due to the absence, for semiring identities, of an analogue of Corollary~\ref{cor:sem-identity}, whose use simplifies the counting in the semigroup case.

\begin{remark}
If one considers identities whose sides are arbitrary terms built from variables by repeated applications of the operations + and $\cdot$ (such as the expression \eqref{eq:n-foldproduct}, for instance) and defines the size of such an identity as the sum of lengths of the left-hand and right-hand terms, the problem \textsc{Check-Id}$(T_n,+,\cdot)$ becomes $\mathsf{coNP}$-complete for each $n\ge 1$. This follows from a general result of Bloniarz, Hunt, and Rosenkrantz \cite[Theorem 3.1]{BHR84}.
\end{remark}

\subsection{Finite Basis Problem}\label{subsec:finbas} Here we use some concepts from equational logic and universal algebra; we provide the necessary definitions.

The \emph{variety} defined by a set $\Sigma$ of identities is the class of all algebras satisfying every identity from $\Sigma$. A variety is \emph{\fb} [\emph{\nfb}] if it can [respectively, cannot] be defined by a finite set of identities. Given an algebra $\mA$, the variety defined by the set of all identities satisfied by $\mA$ is called the \emph{variety generated by $\mA$} and denoted by $\var\mA$. We say that $\mA$ is \fb{} or \nfb{} if so is the variety $\var\mA$. The \emph{Finite Basis Problem} for a class of algebras is the question of classifying algebras in this class for being finitely or \nfb. For the classes of finite semigroups and finite semirings, the Finite Basis Problem has been intensely studied since the 1960s and, respectively, 2000s. We refer to \cite{GSV25} and \cite{JRZ22} as recent samples of advances in these areas.

A variety is said to be \emph{locally finite} if each of its finitely generated members is finite. A finite algebra $\mA$ is called \emph{inherently \nfb} if $\mA$ is not contained in any finitely based locally finite variety. It is well known that the variety generated by a finite algebra is locally finite (see, e.g., \cite[Theorem II.10.16]{BuSa81}). Hence, a finite algebra $\mA$ is \nfb{} whenever the variety $\var\mA$ contains an inherently \nfb\ algebra. In particular, an inherently \nfb\ algebra is \nfb.

A word $\mathbf{w}$ is called an \emph{isoterm} for a semigroup $(S,\cdot)$ if the only word $\mathbf{w'}$ such that $(S,\cdot)$ satisfies the identity $\mathbf{w}\approx\mathbf{w'}$ is the word $\mathbf{w}$ itself. Observe that if a word $\mathbf{w}$ is minimal or maximal for an ordered semigroup $(S,\cdot,\le)$, then $\mathbf{w}$ is an isoterm for the semigroup $(S,\cdot)$.

Recall a combinatorial characterization of inherently \nfb\ semigroups due to Mark Sapir~\cite{Sa87}. It uses the sequence $\{Z_m\}_{m=1,2,\dots}$ of \emph{Zimin words} defined inductively by $Z_1:=x_1$, $Z_{m+1}:=Z_mx_{m+1}Z_m$ where $x_1,x_2,\dots,x_m,\dots$ are distinct variables.

\begin{proposition}[\!{\mdseries\cite[Proposition 7]{Sa87}}]
\label{prop:ziminsemigroup}
A finite semigroup $(S,\cdot)$ is inherently \nfb\ if and only if all Zimin words are isoterms for $(S,\cdot)$.
\end{proposition}

Our description of inequalities of the ordered semigroup $(T_n,\cdot,\le)$ allows to establish the following property of Zimin words.

\begin{proposition}
\label{prop:ziminminimal}
Zimin words are minimal for the ordered semigroup $(T_n,\cdot,\le)$ if $n\ge 3$.
\end{proposition}

\begin{proof} Clearly, the map $T_{n-1}\to T_n$ that sends each matrix $A\in T_{n-1}$ to the $n\times n$-matrix
\[
\begin{pmatrix}&&&0\\&A&&\vdots\\&&&0\\0&\cdots&0&1\end{pmatrix}
\]
is an embedding of ordered semigroups. Therefore, it suffices to show that each Zimin word $Z_m$ is minimal for $(T_3,\cdot,\le)$.

We induct on $m$. When $m=1$, we have $Z_1=x_1$. If for some word $\mathbf{w}$, the inequality $\mathbf{w}\preccurlyeq x_1$ holds in $(T_3,\cdot,\le)$, then by Theorem~\ref{thm:sem-inequality}, either $\mathbf{w}=x_1$ or $|\mathbf{w}|\ge3$ and every word of length 2 that occurs in $\mathbf{w}$ must occur in $x_1$. The latter case is obviously impossible.

Now assume that $m>1$ and the Zimin word $Z_{m-1}$ is minimal for $(T_3,\cdot,\le)$. Consider the Zimin word $Z_m$ and let $\mathbf{w}$ be any word such that the inequality $\mathbf{w}\preccurlyeq Z_m$ holds in $(T_3,\cdot,\le)$. Substituting the identity matrix for the variable $x_1$ in the inequality, we get that $(T_3,\cdot,{\le})$ satisfies the inequality $\mathbf{w'}\preccurlyeq Z'_m$ where the words $\mathbf{w'}$ and $Z'_m$ are obtained from  $\mathbf{w}$ and, respectively, $Z_m$ by removing all occurrences of $x_1$. The word $Z'_m$ is nothing but the Zimin word $Z_{m-1}$ with the subscript of each variable increased by 1. By the induction assumption, we have $\mathbf{w'}=Z'_m$. Since the word $\mathbf{w}$ becomes $Z'_m$ after removing all occurrences of $x_1$, we conclude that
\[
\mathbf{w}=x_1^{\varepsilon_1}x_2x_1^{\varepsilon_2}x_3x_1^{\varepsilon_3}x_2x_1^{\varepsilon_4}x_4\cdots x_3x_1^{\varepsilon_{2^m-1}}x_2x_1^{\varepsilon_{2^m}},
\]
where $\varepsilon_i$, $i=1,2,\dots,2^m$, are nonnegative integers. Thus, the word $\mathbf{w}$ is obtained from $Z_m$ by substituting the factor $x_1^{\varepsilon_i}$ for the $i$-th from the left occurrence of the variable $x_1$.

If $\varepsilon_i>1$ for some $i=1,2,\dots,2^m$, then the word $x_1^2$ of length 2 occurs in $\mathbf{w}$ with gaps $G_1,G_2,G_3$ such that $G_2=\varnothing$. Since the inequality $\mathbf{w}\preccurlyeq Z_m$ holds in $(T_3,\cdot,\le)$, Theorem~\ref{thm:sem-inequality} ensures that the word $x_1^2$ occurs in $Z_m$ with gaps $G'_1,G'_2,G'_3$ such that $G'_\ell\subseteq G_\ell$ for each $\ell=1,2,3$. Then $G'_2=\varnothing$ which would mean that $x_1^2$ is as a factor of $Z_m$ but by the construction of Zimin words, every two occurrences of the variable $x_1$ in $Z_m$ are separated by another variable. Thus, the option $\varepsilon_i>1$ is excluded.

If $\varepsilon_i=0$ for some $i=2,\dots,2^m-1$, then for some $x_j$ and $x_k$ with $j,k>1$, the word $x_jx_k$ of length 2 occurs in $\mathbf{w}$ with gaps $G_1,G_2,G_3$ such that $G_2=\varnothing$. As in the preceding paragraph, we conclude that $x_jx_k$ must occur in $Z_m$ with gaps $G'_1,G'_2,G'_3$ such that the ``inner'' gap $G'_2$ is empty. It would mean that $x_jx_k$ is a factor of $Z_m$ but in $Z_m$, the variable $x_1$ appears between any two occurrences of variables with subscripts greater than 1.

If $\varepsilon_1=0$, the word $x_2x_1$ of length 2 occurs in $\mathbf{w}$ with gaps $G_1,G_2,G_3$ such that $G_1=\varnothing$. Then it must occur in $Z_m$ with gaps $G'_1,G'_2,G'_3$ such that $G'_1=\varnothing$, which would mean that $Z_m$ starts with $x_2$ while it starts with $x_1$. A symmetric argument excludes the option $\varepsilon_{2^m}=0$.

We have shown that $\varepsilon_i=1$ for each $i=1,2,\dots,2^m$, and hence $\mathbf{w}=Z_m$.
\end{proof}

Proposition~\ref{prop:ziminminimal} implies that all Zimin words are isoterms for the semigroup $(T_n,\cdot)$ if $n\ge 3$. Combining this observation with Proposition~\ref{prop:ziminsemigroup} yields the following:
\begin{corollary}
\label{cor:infb}
For each $n\ge 3$, the semigroup of all Boolean upper triangular $n\times n$-matrices is inherently \nfb.
\end{corollary}

While Proposition~\ref{prop:ziminminimal} appears new, Corollary~\ref{cor:infb} is known. For $n>3$, the result was obtained in \cite{VG04}, and the case $n=3$ was settled in \cite{LL11}. Both \cite{VG04} and \cite{LL11} utilized a semantic approach that did not require characterizing the identities of $(T_n,\cdot)$.

For the sake of completeness, we mention that the semigroup $(T_2,\cdot)$ is \fb; an explicit finite basis for the identities of $(T_2,\cdot)$ is found in~\cite[Theorem 3.5]{LL11}.

What can be said about the Finite Basis Problem for the \ais{} $(T_n,+,\cdot)$? A (partial) semiring analog of Proposition~\ref{prop:ziminsemigroup} was found by Dolinka~\cite{Do09}.

\begin{proposition}[\!{\mdseries\cite[Theorem B]{Do09}}]
\label{prop:ziminsemiring}
A finite \ais{} $(S,+,\cdot)$ with $0$ is inherently \nfb\ if each Zimin word is both minimal and maximal for $(S,+,\cdot)$.
\end{proposition}

Because of Proposition~\ref{prop:ziminminimal}, one could have hoped to apply Proposition~\ref{prop:ziminsemiring} to the \ais{} $(T_n,+,\cdot)$ with $n\ge 3$ by showing that all Zimin words are maximal for the ordered semigroup $(T_n,\cdot,\le)$ with $n\ge 3$. However, in reality, for each $n$, there exists a Zimin word that is not maximal for $(T_n,\cdot,\le)$. For instance, using Theorem~\ref{thm:sem-inequality}, it is easy to verify that the strict inequality $Z_3\prec x_1x_2x_1^2x_3 x_1x_2x_1$ holds in the ordered semigroup $(T_3,\cdot,\le)$.

Since Proposition~\ref{prop:ziminsemiring} (unlike Proposition~\ref{prop:ziminsemigroup}) is only a sufficient condition for being inherently \nfb, the fact that Zimin words, in general, are not maximal for $(T_n,\cdot,\le)$ does not exclude the possibility that the \ais{}s $(T_n,+,\cdot)$ are inherently \nfb\ for sufficiently large $n$. Currently, we only know that the \ais{} $(T_3,+,\cdot)$ is \nfb. This is a recent result due to Sergey Gusev and the author~\cite{GV25} which is achieved via a semantic approach.

As in Subsect.~\ref{subsec:complexity}, we want to put the results just discussed in a broader context. For this, we report the state-of-the-art on the Finite Basis Problem for other semigroups of triangular matrices. In the tropical case, it is known that the semigroup of all upper triangular tropical $n\times n$-matrices is \nfb{} for $n=2$ \cite{CHLS16} and $n=3$ \cite{HZL21} while for $n\ge 4$, the problem remains open. The current situation with the semigroup $(T_n(q),\cdot)$ of all upper triangular $n\times n$-matrices over a finite field with $q$ elements is the opposite in a sense. In~\cite{VG03}, it was proved that the semigroup $(T_n(q),\cdot)$ is inherently \nfb{} if and only if $q>2$ and $n\ge4$, and in~\cite{GSV25}, the semigroup $(T_n(2),\cdot)$ with $n\ge4$ was shown to be \nfb. The Finite Basis Problem for the semigroups $(T_2(q),\cdot)$ and $(T_3(q),\cdot)$ remains open with the only exception of the semigroup $(T_2(2),\cdot)$ for whose identities an explicit finite basis was found in~\cite[Theorem 3.1]{ZhLL12}.

\subsection{Connections with formal language theory}\label{subsec:languages}
A \emph{monoid} is a semigroup possessing an identity element. All matrix semigroups considered above were in fact monoids as each of them contained the identity matrix of an appropriate size. The \emph{pseudovariety} generated by a set of finite monoids is the least class of finite monoids containing the set and closed under taking homomorphic images, submonoids and finite direct products. Pseudovarieties of finite monoids are actively studied because of their tight connections to certain classes of recognizable languages via the Eilenberg correspondence (\!\cite{Ei76}, see also~\cite{Pi86}). The pseudovariety generated by a single finite monoid is known to consist of all finite monoids satisfying a certain system of identities~\cite[Corollary II.1.3]{Pi86}. In other words, such a pseudovariety is nothing but the set of all finite members of a monoid variety.

Let $\Sigma$ be a finite set of variables. Denote by $\Sigma^*$ the set of all words whose variables come from $\Sigma$, including the empty word. Under concatenation, $(\Sigma^*,\cdot)$ is a monoid, with the empty word serving as an identity element. A \emph{language over} $\Sigma$ is any subset of $\Sigma^*$. By a \emph{Boolean combination} of a family $\{L_i\}$ of languages over $\Sigma$ we mean any language that can be produced from $L_i$'s by repeated applications of the set-theoretical operations of union or intersection of two sets and taking the complement of a set. For two languages $L,K\subseteq\Sigma^*$ and a variable $x\in\Sigma$, we write $LxK$ for the language $\{\mathbf{u}x\mathbf{v} \mid \mathbf{u}\in L,\,\mathbf{v}\in K\}$.

A language $L\subseteq\Sigma^*$ is \emph{recognized} by a finite monoid $(M,\cdot)$ if there exist a monoid homomorphism $\varphi\colon\Sigma^*\to M$ and a subset $P\subseteq M$ such that $L$ equals the preimage $P\varphi^{-1}$ of $P$ under $\varphi$. A language is called \emph{recognizable} if it is recognized by some finite monoid.

Let $\mathbb{T}$ stand for the pseudovariety of finite monoids generated by all monoids of upper triangular Boolean $n\times n$-matrices, $n=1,2,\dotsc$. Pin and Straubing \cite{PiSt85} found the following combinatorial characterization of languages recognized by monoids in $\mathbb{T}$.
\begin{proposition}[\!{\mdseries\cite[Theorem 2]{PiSt85}}]
\label{prop:pinstraubing}
A language $L$ over a finite set\/ $\Sigma$ is recognized by a monoid in the pseudovariety $\mathbb{T}$ if and only if $L$ is a Boolean combination of languages of the form
\[
\Sigma_1^*u_1\Sigma_2^*\cdots u_k\Sigma_{k+1}^*
\]
where $k\ge1$, $u_1,\dots,u_k\in\Sigma$ and\/ $\Sigma_1,\Sigma_2,\dots,\Sigma_{k+1}$ are \textup(possibly empty\textup) subsets of\/ $\Sigma$.
\end{proposition}

For each $n=1,2,\dotsc$, denote by $\mathbb{T}_n$ the pseudovariety of finite monoids generated by the monoid $(T_n,\cdot)$. Theorem~\ref{thm:ais-inequality} easily implies a similar to Proposition~\ref{prop:pinstraubing} characterization of languages recognized by monoids in $\mathbb{T}_n$.
\begin{proposition}
\label{prop:local}
A language $L$ over a finite set\/ $\Sigma$ is recognized by a monoid in the pseudovariety $\mathbb{T}_n$ if and only if $L$ is a Boolean combination of languages of the form
\begin{equation}\label{eq:language}
\Sigma_1^*u_1\Sigma_2^*\cdots u_k\Sigma_{k+1}^*
\end{equation}
where $1\le k<n$, $u_1,\dots,u_k\in\Sigma$ and\/ $\Sigma_1,\Sigma_2,\dots,\Sigma_{k+1}$ are \textup(possibly empty\textup) subsets of\/ $\Sigma$.
\end{proposition}

First, we restate the semigroup part of Theorem~\ref{thm:ais-inequality} in terms of languages \eqref{eq:language}.

\begin{lemma}
\label{lem:restate}
The ordered semigroup $(T_n,\cdot,\le)$ satisfies an inequality $\mathbf{w}\preccurlyeq\mathbf{w'}$ if and only if the word $\mathbf{w'}$ belongs to every language of the form \eqref{eq:language} that contains the word $\mathbf{w}$.
\end{lemma}

\begin{proof}
\emph{Necessity.} The containment $\mathbf{w}\in\Sigma_1^*u_1\Sigma_2^*\cdots u_k\Sigma_{k+1}^*$ means that $\mathbf{w}$ decomposes as
\[
\mathbf{w} = \mathbf{w}_1u_1\mathbf{w}_2\cdots \mathbf{w}_ku_k\mathbf{w}_{k+1}
\]
with $\mathbf{w}_\ell\in\Sigma_\ell^*$ for each $\ell=1,2,\dots,k+1$. Then each gap $G_\ell=\alf(\mathbf{w}_\ell)$ is contained in $\Sigma_\ell$. If the inequality $\mathbf{w}\preccurlyeq\mathbf{w'}$ holds in $(T_n,\cdot,\le)$, then by Theorem~\ref{thm:ais-inequality}, there is an occurrence of $u_1\cdots u_k$ as a subword in the word $\mathbf{w'}$ with gaps $G'_1,G'_2,\dots,G'_{k+1}$ such that $G'_\ell\subseteq G_\ell$ for~each $\ell=1,2,\dots,k+1$, that is, a decomposition
\[
\mathbf{w'} = \mathbf{w'}_1u_1\mathbf{w'}_2\cdots \mathbf{w'}_ku_k\mathbf{w'}_{k+1}
\]
with $G'_\ell=\alf(\mathbf{w'}_\ell)$. Since $G'_\ell\subseteq G_\ell\subseteq\Sigma_\ell$, each factor $\mathbf{w'}_\ell$ lies in $\Sigma_\ell^*$ for $\ell=1,2,\dots,k{+}1$, and we conclude that $\mathbf{w'}\in\Sigma_1^*u_1\Sigma_2^*\cdots u_k\Sigma_{k+1}^*$. Thus, $\mathbf{w'}$ belongs to every language of the form \eqref{eq:language} containing $\mathbf{w}$.

\medskip

\noindent\emph{Sufficiency.} Suppose that the inequality $\mathbf{w}\preccurlyeq\mathbf{w'}$ fails in $(T_n,\cdot,\le)$. By Theorem~\ref{thm:ais-inequality}, this means that there exists a word $\mathbf{u}$ of length $k<n$ that distinguishes between the words $\mathbf{w}$ and $\mathbf{w'}$ as follows. In $\mathbf{w}$, the subword $\mathbf{u}$ occurs with gaps $G_1,G_2,\dots,G_{k+1}$, while in $\mathbf{w'}$, either no occurrence of $\mathbf{u}$ as a subword exists or, for every such occurrence, there is an index $\ell$ such that the $\ell$-th gap of that occurrence (counted from the left) is not contained in~$G_\ell$. In any case, $\mathbf{w}$ lies in the language $G_1^*u_1G_2^*\cdots u_kG_{k+1}^*$ of the form \eqref{eq:language}, where $u_\ell$ is the variable in the $\ell$-th position of $\mathbf{u}$, whereas $\mathbf{w'}$ does not belong to this language.
\end{proof}

\begin{proof}[Proof of Proposition~\ref{prop:local} ]
Denote the class of all Boolean combinations of languages of the form \eqref{eq:language} by $\mathcal{F}_n$.

\medskip

\noindent\emph{Necessity.} It is well known (and easy to verify) that for any pseudovariety $\mathbb{P}$ of finite monoids, the class of languages over $\Sigma$ recognized by monoids in $\mathbb{P}$ is closed under finite unions, finite intersections, and complementation. Therefore, to prove that every language in $\mathcal{F}_n$ is recognized by a monoid in $\mathbb{T}_n$, it suffices to establish this for languages of the form \eqref{eq:language}. We will show that the language \eqref{eq:language} is recognized by the monoid $(T_n,\cdot)$.

Let $\mathbf{u}:=u_1\cdots u_k$ and fix a word $\mathbf{v}$ of the form
\[
\mathbf{v} = \mathbf{v}_1u_1\mathbf{v}_2\cdots \mathbf{v}_ku_k\mathbf{v}_{k+1}
\]
where for each $\ell=1,2,\dots,k+1$, the word $\mathbf{v}_\ell$ is chosen so that $\alf(\mathbf{v}_\ell)=\Sigma_\ell$. Lemma~\ref{lem:substitution} implies that a word $\mathbf{w}$ belongs to the language \eqref{eq:language} if and only if $\bigl(\mathbf{w}\varphi_{\mathbf{u},\mathbf{v}}\bigr)_{1\,k+1}=1$. Hence, the language \eqref{eq:language} is the preimage of the set $\left\{\bigl(\alpha_{ij}\bigr)\in T_n\mid \alpha_{1\,k+1}=1\right\}$ under $\varphi_{\mathbf{u},\mathbf{v}}$.

\medskip

\noindent\emph{Sufficiency.} Let $(F,\cdot)$ be the $\mathbb{T}_n$-free monoid with the set $\Sigma$ of free generators; see \cite[Sections II.10 and II.11]{BuSa81} for a general discussion of free objects and their basic features. We will use the following properties of the free monoid all of which are specializations of well-known universal-algebraic facts.
\begin{enumerate}
  \item[$0^\circ$] The identity map on $\Sigma$ uniquely extends to a monoid homomorphism $\eta\colon\Sigma^*\to F$.
  \item[$1^\circ$] For all words $\mathbf{w},\mathbf{w'}\in\Sigma^*$, the identity $\mathbf{w}\approx\mathbf{w'}$ holds in $(T_n,\cdot)$ if and only if $\mathbf{w}\eta=\mathbf{w'}\eta$ (specialization of \cite[Theorem II.11.4]{BuSa81}).
  \item[$2^\circ$] For each monoid $(M,\cdot)$ in $\mathbb{T}_n$, every map $\Sigma\to M$ uniquely extends to a monoid homomorphism $F\to M$ (specialization of \cite[Theorem II.10.10]{BuSa81}).
  \item[$3^\circ$] The monoid $(F,\cdot)$ is finite (specialization of \cite[Theorem II.10.16]{BuSa81}).
\end{enumerate}

Now take any language $L$ over $\Sigma$ recognized by a monoid $(M,\cdot)$ in the pseudovariety $\mathbb{T}_n$. By definition, this means that $L=P\varphi^{-1}$ for some homomorphism $\varphi\colon\Sigma^*\to M$ and a subset $P\subseteq M$. By property $2^\circ$, there is a a monoid homomorphism $\psi\colon F\to M$ extending the map $\Sigma\to M$ given by $x\mapsto x\varphi$. Then $\varphi =\eta\psi$ and $L=P\varphi^{-1}=Q\eta^{-1}$, where $Q:=P\psi^{-1}$. The set $Q$ is finite (property $3^\circ$), and since $L=Q\eta^{-1}=\bigcup_{q\in Q}q\eta^{-1}$ and the class $\mathcal{F}_n$ is closed under finite unions, it suffices to show that $q\eta^{-1}\in\mathcal{F}_n$ for every $q\in F$.

By the definition, $q\eta^{-1}=\{\mathbf{w}\in\Sigma^*\mid \mathbf{w}\eta=q\}$. By property $1^\circ$, the identity $\mathbf{w}\approx\mathbf{w'}$ holds in $(T_n,\cdot)$ for all words $\mathbf{w},\mathbf{w'}\in q\eta^{-1}$.
Lemma~\ref{lem:restate} implies that $\mathbf{w}$ and $\mathbf{w'}$ lie in exactly the same languages of the form \eqref{eq:language}. Since $k<n$ and the set $\Sigma$ is finite, there are only finitely many choices for $u_1,\dots,u_k$ and for $\Sigma_1,\Sigma_2,\dots,\Sigma_{k+1}$; therefore the set of languages of the form \eqref{eq:language} is finite. Let $L_1,\dots,L_s$ be all languages of the form \eqref{eq:language} that contain a word in $q\eta^{-1}$ (and hence contain every word in in $q\eta^{-1}$). By Lemma~\ref{lem:restate}, the language $\bigcap_{p=1}^sL_p$ equals the set of all words $\mathbf{v}$ such that ordered semigroup $(T_n,\cdot,\le)$ satisfies each inequality $\mathbf{w}\preccurlyeq\mathbf{v}$ where $\mathbf{w}\in q\eta^{-1}$.  For every $\mathbf{v}\in\bigcap_{p=1}^sL_p\setminus q\eta^{-1}$ and every $\mathbf{w}\in q\eta^{-1}$, the inequality $\mathbf{v}\preccurlyeq\mathbf{w}$ fails in $(T_n,\cdot,\le)$. By Lemma~\ref{lem:restate}, there exists a language  of the form \eqref{eq:language} that contains $\mathbf{v}$ but omits $\mathbf{w}$. Let $K_1,\dots,K_t$ be all languages of the form \eqref{eq:language} that omit a word in $q\eta^{-1}$ (and hence omit every word in $q\eta^{-1}$). Then the language $\bigcup_{r=1}^t{K}_r$ is disjoint with $q\eta^{-1}$ and contains every word in $\bigcap_{p=1}^sL_p\setminus q\eta^{-1}$. Using de Morgan law, we conclude that
\[
q\eta^{-1}=\bigcap_{p=1}^sL_p\setminus\bigcup_{r=1}^t{K}_r=\bigcap_{p=1}^sL_p\cap\bigcap_{r=1}^t\overline{K}_r,
\]
where $\overline{K}_r$ denotes the complement of the language $K_r$. Since the class $\mathcal{F}_n$ is closed under finite intersections and complementation, we obtain $q\eta^{-1}\in\mathcal{F}_n$.
\end{proof}

From the definition a pseudovariety, one readily sees that $\mathbb{T}=\bigcup_n\mathbb{T}_n$. Therefore, one can view Proposition~\ref{prop:local} as a refinement of the quoted result by Pin and Straubing, and moreover, Proposition~\ref{prop:pinstraubing} is an immediate consequence of Proposition~\ref{prop:local}. However, it is fair to say that our arguments in the proof of Theorem~\ref{thm:ais-inequality} stem from \cite{PiSt85}.

We have presented Proposition~\ref{prop:local} as a combinatorial characterization of languages defined algebraically as those recognized by monoids in the pseudovariety $\mathbb{T}_n$. From the viewpoint of language theory, it may also be regarded as an algebraic characterization of the class of all Boolean combinations of languages of the form \eqref{eq:language}, thereby providing a solution to a problem posed by Kl\'ima and Pol\'ak in \cite[Remark~4]{KP10}.

Several authors modified the Eilenberg correspondence between classes of recognizable languages and pseudovarieties of finite monoids by relaxing conditions imposed on language classes while enriching monoids with extra relations and/or operations. In particular, Pin~\cite{Pi95} and Pol\'ak~\cite{Po01} developed Eilenberg-type correspondences with ordered monoids and, respectively, \ais{}s on the algebraic side. Using Theorem~\ref{thm:ais-inequality}  allows one to characterize languages recognized by ordered monoids and \ais{} from the pseudovarieties of ordered monoids and \ais{}s generated by the ordered monoid $(T_n,\cdot,\le)$ and, respectively, the \ais{} $(T_n,+,\cdot)$ for each $n=1,2,\dotsc$. The characterizations are similar to that of Proposition~\ref{prop:local} so we omit their formulations.

\subsection*{Acknowledgement} The author is grateful to the anonymous referee for their feedback and helpful suggestions.


\small

\begin{thebibliography}{11}
\bibitem{AVG09}
Almeida, J., Volkov, M.V., Goldberg, S.V. Complexity of the identity checking problem for finite semigroups. Zap. Nauchn. Sem. POMI \textbf{358}, 5--22 (2008) [Russian; Engl.\ translation J. Math. Sci. \textbf{158}(5), 605--614 (2009)]

\bibitem{BHR84}
Bloniarz, P.A., Hunt III, H.B., Rosenkrantz, D.J.: Algebraic structures with hard equivalence and minimization problems. J. Assoc. Comput. Mach. \textbf{31}, 879--904 (1984)

\bibitem{BuSa81}
Burris, S., Sankappanavar, H.P.: A Course in Universal Algebra. Springer, Berlin, Heidelberg, New York (1981)

\bibitem{CHLS16}
Chen, Y.Z., Hu, X., Luo, Y.F., Sapir, O.: The finite basis problem for the monoid of two-by-two upper triangular tropical matrices. Bull. Aust. Math. Soc. \textbf{94}, 54--64 (2016)

\bibitem{DJK18}
Daviaud, L., Johnson, M., Kambites, M.: {Identities in upper triangular tropical matrix semigroups and the bicyclic monoid}. J. Algebra \textbf{501}, 503--525 (2018)

\bibitem{Do09}
Dolinka, I.: A class of inherently nonfinitely based semirings. Algebra Universalis \textbf{60}, 19--35 (2009)

\bibitem{Ei76}
Eilenberg, S.: Automata, Languages and Machines, Vol.\,B. Academic Press, New York (1976)

\bibitem{GSV25}
Gusev, S.V., Sapir, O.B., Volkov, M.V.: Strongly nonfinitely based monoids. J. Combinatorial Algebra \textbf{9}(1-2), 129--144 (2025)

\bibitem{GV25}
Gusev, S.V., Volkov, M.V.: The finite basis problem for semirings of triangular Boolean matrtices. In preparation.

\bibitem{HZL21}
Han, B.B., Zhang, W.T., Luo, Y.F.: Equational theories of upper triangular tropical matrix semigroups. Algebra Universalis \textbf{82}, article no. 44 (2021)

\bibitem{JRZ22}
Jackson, M., Ren, M., Zhao, X.Z.: Nonfinitely based ai-semirings with finitely based semigroup reducts. J. Algebra \textbf{611}, 211--245 (2022)

\bibitem{JZ21}
Jackson, M., Zhang, W.T.: From $A$ to $B$ to $Z$. Semigroup Forum \textbf{103}, 165--190 (2021)

\bibitem{JT19}
Johnson, M., Tran, N.M.: Geometry and algorithms for upper triangular tropical matrix identities. J. Algebra \textbf{530}, 470--507 (2019)

\bibitem{Ka72}
Karp, R.M.: Reducibilty among combinatorial problems.  In: Miller, R.E., Thatcher, J.W., Bohlinger, J.D. (eds.), Complexity of Computer Computations, pp. 85--103. Springer, Boston (1972)

\bibitem{KS95}
Kharlampovich, O.G., Sapir, M.V.: Algorithmic problems in varieties. Int. J. Algebra Comput. \textbf{5}(4-5), 379--602 (1995)

\bibitem{Kl09}
Kl\'{\i}ma, O.: Complexity issues of checking identities in finite monoids. Semigroup Forum \textbf{79}, 435--444 (2009)

\bibitem{KP10}
Kl\'ima, O., Pol\'ak, L.: Hierarchies of piecewise testable languages. Int. J. Found. Comput. Sci. \textbf{21}(4), 517--533  (2010)

\bibitem{LL11}
Li, J.R., Luo, Y.F.: On the finite basis problem for the monoids of triangular boolean matrices. Algebra Universalis \textbf{65}, 353--362 (2011)

\bibitem{Pa94}
Papadimitriou, C.H.: {Computational Complexity}. Addison-Wesley, Reading, MA  (1994)

\bibitem{Pi86}
Pin, J.-\'E.: Vari\'et\'es de Langages Formels. Masson, Paris (1984) [French; Engl.\ translation: Varieties of Formal Languages. North Oxford Academic, London (1986) and Plenum, New York (1986)]

\bibitem{Pi95}
Pin, J.-\'E.: A variety theorem without complementation. Izv. Vyssh. Uchebn. Zav. Mat. 1, 80--90 (1995) [Russian; Engl.\ translation Russian Math. (Iz. VUZ) \textbf{39}(1), 74--83 (1995)]

\bibitem{PiSt85}
Pin, J.-\'E., Straubing, H.: Monoids of upper triangular matrices. In: Poll\'ak, Gy.,\ Schwarz, \v{S}t., Steinfeld, O. (eds), Semigroups. Structure and Universal Algebraic Problems. Colloquia Mathematica Societatis
J\'anos Bolyai, vol. 39, pp. 259--272. North-Holland, Amsterdam (1985)

\bibitem{Po01}
Pol\'ak, L.: Syntactic semiring of a language. In: Sgall, J., Pultr, A., Kolman, P. (eds.), Mathematical Foundations of Computer Science 2001. MFCS 2001. Lecture Notes in Computer Science, vol. 2136, pp. 611--620. Springer, Berlin, Heidelberg (2001)

\bibitem{Sa87}
Sapir, M.V.: {Problems of Burnside type and the finite basis property in varieties of semigroups}. Izv. Akad. Nauk SSSR, Ser. Mat. \textbf{51}, 319--340 (1987) [Russian; Engl. translation Math. USSR--Izv. \textbf{30}, 295--314 (1988)]

\bibitem{Se05}
Seif, S.: The Perkins semigroup has co-NP-complete term-equivalence problem. Int. J. Algebra Comput. \textbf{15}(2), 317--326 (2005)

\bibitem{Vo04}
Volkov, M.V.: Reflexive relations, extensive transformations and piecewise testable languages of a given height.  Int. J. Algebra Comput. \textbf{14}(5-6),  817--827 (2004)

\bibitem{Vo:embedding}
Volkov, M.V.: A new Boolean matrix representation for Catalan semirings. Preprint arXiv:2510.05045 (2025). \url{https://doi.org/10.48550/arXiv.2510.05045}


\bibitem{VG03}
Volkov, M.V., Goldberg, I.A.: Identities of semigroups of triangular matrices over finite fields. Mat. Zametki \textbf{73}(4), 502--510 (2003) [Russian; Engl. translation Math. Notes \textbf{73}(4), 474--481 (2003)]

\bibitem{VG04}
Volkov, M.V., Goldberg, I.A.: The finite basis problems for monoids of triangular boolean matrices. In: Algebraic Systems, Formal Languages, and Conventional and Unconventional Computation Theory. RIMS Kokyuroku, vol. 1366, pp. 205--214. Kyoto University, Kyoto (2004)

\bibitem{ZhLL12}
Zhang, W.T., Li, J.R., Luo, Y.F.: On the variety generated by the monoid of triangular $2\times2$ matrices over a two-element field. Bull. Aust. Math. Soc. \textbf{86}(1), 64--77 (2012)
\end{thebibliography}
\end{document}